\documentclass [11pt,twoside,a4paper]{article}
\usepackage{mathrsfs}
\usepackage{amssymb}
\usepackage{amsthm}
\usepackage{amsfonts}
\usepackage{amsmath}
\usepackage{amstext}
\usepackage{amscd}
\usepackage{dsfont}
\newtheorem{theorem}{Theorem}[section]
\newtheorem{proposition}{Proposition}[section]
\newtheorem{lemma}{Lemma}[section]
\newtheorem{corollary}{Corollary}[section]
\newtheorem{remark}{Remark}[section]
\newtheorem{definition}{Definition}[section]

 \def\dsum{\displaystyle\sum}
\def\dfrac{\displaystyle\frac}
\def\M*{M\setminus \{p_1,p_2,\ldots,p_N,q_1,q_2,\ldots,q_S\}}
\def\S*{\Sigma \setminus \{p_1,p_2,\ldots,p_I,q_1,q_2,\ldots,q_J\}}
\def\dps{\displaystyle}
\def\qdl{\textquotedblleft}
\def\ekm{extremal K\"{a}hler metric}
\def\ehm{extremal Hermitian metric}
\def\BE{\begin{equation}}
\def\EE{\end{equation}}
\def\lccc{local complex coordinate chart }
\newcommand{\bey}{\begin{eqnarray}}
\newcommand{\eey}{\end{eqnarray}}
\newcommand{\beyy}{\begin{eqnarray*}}
\newcommand{\eeyy}{\end{eqnarray*}}
\newcommand{\nn}{\nonumber}
\begin{document}

\title{$\ \ $ HCMU metrics with cusp singularities and conical singularities }
\author{Qing Chen\footnotemark[1], \ Yingyi Wu\footnotemark[2], \ Bin Xu\footnotemark[3]}

\date{}
\maketitle

\begin{abstract}
An HCMU metric is a conformal metric which has a finite number of
singularities on a compact Riemann surface and satisfies the
equation of the extremal K\"{a}hler metric. In this paper, we give
a necessary and sufficient condition for the existence of a kind
of HCMU metrics which has both cusp singularities and conical singularities.
\end{abstract}

\footnotetext[1]{The first author is supported in part by the National 
Natural Science Foundation of China (Grant No. 11271343).}

\footnotetext[2]{The second author is supported in part by the Youth Fund (Grant No. 10901160), the
National Natural Science Foundation of China (Grant No. 11071249)
and the President Fund of UCAS.}

\footnotetext[3]{The last author is supported in
part by Anhui Provincial Natural Science Foundation (Grant No.
1208085MA01) and the Fundamental Research Funds for the Central
Universities (Grant No. WK0010000020).}

\section{Introduction}
The extremal K\"{a}hler metric on a K\"{a}hler manifold was defined in \cite{Ca} by Calabi.
The aim is to find the \qdl
best" metric in a fixed K\"{a}hler class on a compact K\"{a}hler manifold
$M$. In a fixed K\"{a}hler class, an extremal K\"{a}hler metric is the critical
point of the following Calabi energy functional
$$
\mathcal{C}(g)=\int_{M} R^2 dg,
$$
where $R$ is the scalar curvature of the metric $g$ in the K\"{a}hler
class. The Euler-Lagrange equation of $\mathcal{C}(g)$ is $R_{,\alpha
\beta}=0$, where $R_{,\alpha\beta}$ is the second-order $(0,2)$
covariant derivative of $R$. When $M$ is a compact Riemann surface without boundary,
 Calabi proved that
an extremal K\"{a}hler metric is a CSC(constant scalar curvature) metric in
\cite{Ca}. This coincides with  the classical uniformization
theorem, which says that there exists a CSC
metric in each fixed K\"{a}hler class of Riemann surface without
boundary.

On the other hand, there have been
many attempts to generalize the classical uniformization theorem to surfaces with boundaries.
The main focus, started by the independent work of Troyanov\cite{T}
and McOwen\cite{M}, has been to study the existence or nonexistence
of constant curvature metrics on surfaces with conical
singularities. But in general one should not expect to get a
clear-cut statement about the existence(or nonexistence) of
solutions, since the constant curvature equation is overdetermined
in this case. Therefore we can consider extremal K\"{a}hler metrics with singularities as the generalization
of constant curvature metrics on Riemann surfaces with conical singularities. In this paper, we study two kinds of
singularities: cusp singularities and conical singularities.

\par
   Now let $\Sigma$ be a compact Riemann surface and $\{a_1,a_2,\cdots,a_n\}$ be
a finite set of $\Sigma$. We call a smooth metric $g$ on $\Sigma\setminus\{a_1,a_2,\cdots,a_n\}$ an
\ehm~(v.s.\cite{Ch2}) if $g$ satisfies
\begin{equation}\label{ehmequ1}
 \Delta_{g}K+K^2=C,
\end{equation}
where $K$ is the Gauss curvature of $g$ and $C$ is a constant. This condition is equivalent to
\begin{equation}\label{ehmequ2}
 \dfrac{\partial K_{,zz}}{\partial \bar{z}}=0.
\end{equation}
One can refer to \cite{Ch3} for details. (\ref{ehmequ2}) has a special case
\begin{equation}\label{HCMUequ}
  K_{,zz}=0.
\end{equation}
We call a metric an HCMU(the Hessian of the Curvature of the Metric is Umbilical) metric
(v.s.\cite{Ch3}) if the metric satisfies (\ref{HCMUequ}).  Obviously an HCMU metric can be regarded as
a direct generalization of an \ekm~to a punctured Riemann surface. {\bf In the following we always assume
that an \ehm~or an HCMU metric has
finite area and finite Calabi energy}, that is,
\begin{equation}\label{finite}
 \int_{\Sigma\setminus\{a_1,a_2,\cdots,a_n\}}dg<+\infty, ~~ \int_{\Sigma\setminus\{a_1,a_2,\cdots,a_n\}}
K^2 dg <+\infty.
\end{equation}
\par
For a general \ehm~$g$ on $\Sigma\setminus\{a_1,a_2,\cdots,a_n\}$, if $g$ has cusp singularities at
$a_1,a_2,\cdots,a_n$, X.X.Chen in \cite{Ch2} proved that $g$ must be an HCMU metric and
gave a classification theorem, and if $g$ has conical singularities at $a_1,a_2,\cdots,a_n$ and at each singularity the singular
angle is less than or equal to $\cfrac{\pi}{2}$, G.F.Wang and X.H.Zhu in \cite{WZ} proved that $g$ is also an HCMU metric and
gave a classification theorem.
\par
For an HCMU metric $g$ on $\Sigma\setminus\{a_1,a_2,\cdots,a_n\}$ which is not a CSC metric, if $g$ has conical singularities at
$a_1,a_2,\cdots,a_n$, the first two authors in \cite{CW2} gave a sufficient and necessary condition for the existence of this kind of
metric, that is,
\begin{theorem}[\cite{CW2}]\label{conicalHCMU}
Let $M$ be a compact Riemann surface and $p_1,p_2,\ldots,p_N$ be $N$
points on $M$. Suppose that  $\alpha_1,\alpha_2,\ldots,\alpha_N$ are
$N$ positive real numbers($\alpha_n\neq 1$, $n=1,2,\ldots,N$) and
$\alpha_1,\alpha_2,\ldots,\alpha_J$ are integers with $\alpha_j \geq
2,j=1,2,\ldots,J$. Then there exists a normalized HCMU metric $g$ on
$M$ such that $g$ has conical singularities at $p_n$ with the angles
$2\pi \alpha_n~(1\le n\le N)$ and $p_1,p_2,\ldots,p_J$ are the
saddle points of the curvature $K$, if and only if
\begin{description}
\item[1.]$S\stackrel{\triangle}{=}\dsum_{j=1}^J\alpha_j+\chi(M)-N\geq
0$,
\item[2.]
there are $S$ distinct points $\{q_1,q_2,\ldots,q_S\}\subset M
\setminus \{p_1,p_2,\ldots,p_N\}$ such that we can choose $L(0\leq
L\leq N-J)$ points in $\{p_{J+1},p_{J+2},\ldots,p_N\}$ (w.l.o.g. we
assume these points are $p_{J+1},p_{J+2},\ldots,p_{J+L}$), and $T(0
\leq T\leq S)$ points in $\{q_1,q_2,\ldots,q_S\}$ (w.l.o.g. we
assume these points are $q_1,q_2,\ldots,q_T$), to satisfy

\begin{description}

\item[i)]$\alpha_{max}\stackrel{\triangle}{=}\dsum_{l=J+1}^{J+L}\alpha_l+T>
\alpha_{min}\stackrel{\triangle}{=}\dsum_{m=J+L+1}^N \alpha_m+S-T>0$,
\item[ii)]there exists  a meromorphic 1-form $\omega$ on $M$ satisfying
\begin{itemize}
\item[(a)]$(\omega)=\dsum_{j=1}^J(\alpha_j-1)P_j-\dsum_{k=J+1}^N
P_k-\dsum_{\xi=1}^S Q_\xi$,

\item[(b)]$Res_{p_l}(\omega)=\sigma\alpha_l$, $l=J+1,J+2,\ldots,J+L$; $Res_{p_m}(\omega)=
\sigma \lambda \alpha_m$, $m=J+L+1,J+L+2,\ldots,N$;
$Res_{q_\mu}(\omega)=\sigma$, $\mu=1,2,\ldots,T$ and
$Res_{q_\nu}(\omega)=\sigma \lambda$, $\nu=T+1,T+2,\ldots,S$, where
$\lambda=-\cfrac{\alpha_{max}}{\alpha_{min}}$ and
$\sigma=-\cfrac{(2\lambda+1)^2}{3\lambda(\lambda+1)}$,

\item[(c)]$\omega+\bar{\omega}$ is exact on $\M*$.
\end{itemize}
\end{description}
\end{description}
\end{theorem}
\noindent In fact from \cite{CW2} we get any HCMU metric which is not a CSC metric and has conical singularities
is determined by the following system:
\BE\label{sys0}
\begin{cases}
 \cfrac{dK}{-\frac{1}{3}(K-K_1)(K-K_2)(K+K_1+K_2)}=\omega+\bar{\omega} \\
g=-\frac{4}{3}(K-K_1)(K-K_2)(K+K_1+K_2)\omega \bar{\omega}\\
K(p_0)=K_0,~K_2<K_0<K_1,~p_0 \in M \setminus \{zeros~and~poles~of~\omega\},
\end{cases}
\EE
where $K_1>0$ which is the maximum of the Gauss curvature $K$, $K_1>K_2>-\cfrac{K_1}{2}$ which is the minimum of
the Gauss curvature $K$ and $\omega$ is a meromorphic 1-form on $M$ with the properties:
\begin{enumerate}
  \item  $\omega$ only has simple poles,
  \item  the residue of $\omega$ at each pole is a real number,
  \item  $\omega+\bar{\omega}$ is exact on $M \setminus \{poles~of~\omega\}$.
\end{enumerate}
However (\ref{sys0}) does not include the case that an HCMU metric has cusp singularities.
Therefore we need to reconsider this case. In this paper we try to give a sufficient and
necessary condition for the existence of an HCMU metric which is not a CSC metric and has cusp singularities and
conical singularities. That is our main theorem:
\noindent \begin{theorem}\label{mainth}
Let $\Sigma$ be a compact Riemann surface and $p_1,p_2,\cdots,p_I,q_1,q_2,\cdots,q_J$ be $I+J$ points on
$\Sigma$($I>0,J\geq 0$). Suppose that $\alpha_1,\alpha_2,\cdots,\alpha_J$ are $J$ positive real numbers
such that $\alpha_j \neq 1,~j=1,2,\cdots,J$, and $\alpha_1,\alpha _2,\cdots,\alpha_L$ are integers($L\leq J$). Then there exists
an HCMU metric $g$ on $\Sigma \setminus \{p_1, p_2,\cdots,p_I,q_1,q_2,\cdots,q_J\}$
which is not a CSC metric such that $g$ has cusp singularities at $p_i,~i=1,2,\cdots,I$, has conical
singularities at $q_j,~j=1,2,\cdots,J$, with the angle $2\pi\alpha_j$ respectively and $q_1,q_2,\cdots,q_L$
are the saddle points of the Gauss curvature $K$, if and only if
\begin{itemize}
 \item[1.] $S:=\dsum_{l=1}^L \alpha_l-I-J+\chi(\Sigma)\geq 0$,
 \item[2.] there are $S$ distinct points $\{e_1, e_2, \cdots, e_S\}\subset \Sigma\setminus\{p_1,p_2,\cdots,p_I,q_1,q_2,
 \cdots,q_J\}$ such that there exists a meromorphic 1-form $\omega$ on $\Sigma$ which satisfies
 \begin{itemize}
  \item[(a)] $(\omega)=\dsum_{l=1}^L (\alpha_l-1)Q_l-\dsum_{i=1}^I P_i-\dsum_{l'=L+1}^J Q_{l'}-\dsum_{s=1}^S E_s$;
  \item[(b)] $Res_{p_i}(\omega)>0,~ i=1,2,\cdots,I$, \\
  $Res_{q_{l'}}(\omega)=\Lambda\alpha_{l'},~l'=L+1,L+2,\cdots,J$, \\
  $Res_{e_s}(\omega)=\Lambda,~s=1,2,\cdots,S$, \\
  where $\Lambda$ is a negative real number;
  \item[(c)] $\omega+\bar{\omega}$ is exact on $\Sigma \setminus \{p_1,p_2,\cdots,p_I,q_1,q_2,
\cdots,q_J,e_1,e_2,\cdots,
  e_S\}$.
 \end{itemize}
\end{itemize}
\end{theorem}
\noindent {\bf Here we declare that in the following any HCMU metric that we mention is not a CSC metric}.
In this paper we can get any HCMU metric with cusp singularities and conical singularities is
 determined by the following system:
 \BE\label{sys0'}
\begin{cases}
 \cfrac{dK}{-\frac{1}{3}(K-\mu)^2(K+2\mu)}=\omega+\bar{\omega} \\
g=-\frac{4}{3}(K-\mu)^2(K+2\mu)\omega \bar{\omega}\\
K(p_0)=K_0,~\mu<K_0<-2\mu,~p_0 \in \Sigma \setminus \{zeros~and~poles~of~\omega\},
\end{cases}
\EE
where $\mu<0$ which is the minimum of the Gauss curvature $K$, $-2\mu$ is the maximum of the Gauss curvature $K$ and
$\omega$ is a meromorphic 1-form on $\Sigma$ with the same properties as the meromorphic 1-form
in (\ref{sys0}). From (\ref{sys0'}) we can obtain that the metric $g$ just has cusp singularities at the poles of
$\omega$ where the residues of $\omega$ are all positive and the Gauss curvature $K$ just has the minimum $\mu$ at
the cusp singularities of $g$.\par

The contents of this paper will be organized as following. In Section 2, we will give definitions of conical singularity and
cusp singularity and review some local results about an extremal Hermitian metric around a singularity. Then in Section 3.1, we will
prove the necessity of the main theorem. In this section, we first study $\nabla K$ of an HCMU metric, then we define the dual
1-form of $\nabla K$ as the character 1-form of the metric, and then using the character 1-form we study the properties of $g$
and $K$ at cusp singularities of $g$, smooth singularities of $\nabla K$ and conical singularities of $g$,
finally we give formulas for integrals of $K^n$ over $\Sigma$, $n=0,1,2,\ldots$. In Section 3.2, we will prove the sufficiency of the main theorem.
In this section, first we study the solution of an ODE,
then we use the solution to construct an HCMU metric which satisfies the given conditions. In Section 4, we will discuss the
existence of the meromorphic 1-form in (\ref{sys0}).

\section{Definitions of singularities and local behaviors of an \ehm}
\begin{definition}[\cite{T}]\label{conedef}
 Let $X$ be a Riemann surface, $p\in X$. A conformal metric $g$ on $X$ is said to have a conical singularity at
 $p$ with the singular angle $2 \pi \alpha(\alpha>0)$ if in a neighborhood of $p$
 \begin{equation}\label{ds^2}
  g=e^{2\varphi}|dz|^2,
 \end{equation}
 where $z$ is a local complex coordinate defined in the neighborhood of $p$ with $z(p)=0$ and
 \begin{equation}\label{coneequ1}
  \varphi-(\alpha-1)\ln|z|
 \end{equation}
 is continuous at $0$.
\end{definition}

\begin{remark}
 By (\ref{ds^2}) and (\ref{coneequ1}), in a neighborhood of a conical singularity $p$, the metric $g$ can also be
 expressed as
 \begin{equation}\label{coneequ2}
  g=\frac{h}{|z|^{2-2\alpha}}|dz|^2,
 \end{equation}
 where $h$ is a positive continuous function in the neighborhood of $p$ and is smooth in the
neighborhood except for the
 origin. By (\ref{coneequ1}), we have the following limit at $p$
 \BE\label{conelim}
  \lim_{z\rightarrow 0} \frac{\varphi+\ln|z|}{\ln|z|}=\alpha.
 \EE
 \end{remark}
\begin{definition}[\cite{Ch2}]\label{weakcuspdef}
Let $X$ be a Riemann surface, $p \in X$. A conformal metric $g$ which has finite area and finte
Calabi energy on $X$ is said to have a weak
cusp singularity at $p$ if in a neighborhood of $p$
\BE
 g=e^{2\varphi} |dz|^2,
\EE
 where $z$ is a local complex coordinate defined in the neighborhood of $p$ with $z(p)=0$ and
 \BE\label{weakcuspequ}
  \liminf_{r\longrightarrow 0}\int_0^{2\pi} r ~\frac{\partial(\varphi+\ln r)}{\partial r}d\theta=0
(z=re^{\sqrt{-1}\theta}).
 \EE
\end{definition}
\begin{definition}\label{cuspdef}
 Let $X$ be a Riemann surface, $p \in X$. A conformal metric $g$ on $X$ is said to have a cusp
singularity at $p$ if
in a neighborhood of $p$
\BE
 g=e^{2\varphi} |dz|^2,
\EE
 where $z$ is a local complex coordinate defined in the neighborhood of $p$ with $z(p)=0$ and
 \BE\label{cusplim}
  \lim_{z\rightarrow 0} \frac{\varphi+\ln|z|}{\ln|z|}=0.
 \EE
\end{definition}
\par
X.X.Chen in \cite{Ch1} proved the following theorem:
\begin{theorem}[\cite{Ch1}]\label{finiteth}
 Let $g=e^{2\varphi}|dz|^2$ be a metric on a punctured disk $D \setminus \{0\}$ with finite area and finite Calabi
 energy. Define $\phi(r)=\frac{1}{2\pi}\int_{0}^{2\pi} \varphi(r\cos \theta, r \sin \theta)d\theta$. The following
 three statements hold true:
 \begin{itemize}
  \item[(1)] $\dps{\lim_{r\rightarrow 0}} (\varphi+\ln r)=-\infty$.
  \item[(2)] $\dps{\lim_{r\rightarrow 0}}\phi'(r)r$ exists and is finite.
  \item[(3)] There exist a constant $\beta \in (0,1)$ and two constants $C_1$ and $C_2$  such that
  $$ \frac{1}{\beta} (\phi(r)+\ln r)+C_1 \leq \varphi+\ln r\leq \beta(\phi(r)+\ln r)+C_2.
  $$
 \end{itemize}
\end{theorem}
\noindent By Theorem \ref{finiteth}, we have the following corollary:
\begin{corollary}\label{finitecor}
 Let $g=e^{2\varphi}|dz|^2$ be a metric on a punctured disk $D \setminus \{0\}$ with finite area and finite
 Calabi energy. Then the following two statements are equivalent to each other:
 \begin{itemize}
 \item[(1)] $g$ has a cusp singularity at $0$.
 \item[(2)] $g$ has a weak cusp singularity at $0$.
  \end{itemize}
\end{corollary}
\begin{proof}(1)$\Longrightarrow$(2): If $g$ has a cusp singularity at $0$, that is,
 $ \dps{\lim_{r\rightarrow 0} \frac{\varphi+\ln r}{\ln r}=0},
 $
 then
 $$\dps{ \lim_{r\rightarrow 0}\cfrac{\frac{1}{2\pi}\int_0^{2\pi}(\varphi+\ln r) d\theta}{\ln r}=0},~\text{i.e.}~
 \dps{\lim_{r\rightarrow 0}\frac{\phi(r)+\ln r}{\ln r}=0}.
 $$
 By (1) in Theorem \ref{finiteth}, $\dps{\lim_{r\rightarrow 0}}(\phi(r)+\ln r)=-\infty$.
By (2) in Theorem \ref{finiteth}, the limit
 $$\dps{\lim_{r\rightarrow 0}\frac{(\phi(r)+\ln r)'}{(\ln r)'}=\lim_{r\rightarrow 0}(\phi'(r)r+1)}$$ exists and is finite.
 Then by L'Hospital's Rule, $\dps{\lim_{r\rightarrow 0}(\phi'(r)r+1)=\lim_{r\rightarrow 0}\frac{\phi(r)+\ln r}{\ln r}=0}$,
 which implies that $0$ is a weak cusp singularity of $g$. \\
 (2)$\Longrightarrow$(1): First by (2) in Theorem \ref{finiteth} and L'Hospital's Rule,
 \[ \lim_{r\rightarrow 0} \frac{\phi(r)+\ln r}{\ln r}=\lim_{r\rightarrow 0}(\phi'(r)r+1).
 \]
Since $0$ is a weak cusp singularity of $g$, $\dps{\liminf_{r\longrightarrow 0}\int_0^{2\pi} r ~\frac{\partial(\varphi+\ln r)}
{\partial r}d\theta=0}$, which means
\[ \lim_{r\rightarrow 0} \frac{\phi(r)+\ln r}{\ln r}=\lim_{r\rightarrow 0}(\phi'(r)r+1)=0.
\]
By (3) in Theorem \ref{finiteth}, there exist $\beta \in (0,1)$ and two constants $C_1$ and
$C_2$ such that
\[ \frac{1}{\beta} (\phi(r)+\ln r)+C_1 \leq \varphi+\ln r\leq \beta(\phi(r)+\ln r)+C_2, \]
so \[ \frac{1}{\beta} \frac{\phi(r)+\ln r}{\ln r}+\frac{C_1}{\ln r} \geq \frac{\varphi+\ln r}{\ln r}
\geq \beta \frac{\phi(r)+\ln r}{\ln r}+\frac{C_2}{\ln r}.
\]
Then $\dps{\lim_{r\rightarrow 0}\frac{\varphi+\ln r}{\ln r}=0}$, which shows $0$ is a cusp singularity of $g$.
We prove the corollary.
\end{proof}

Then X.X.Chen in \cite{Ch2} proved the theorem:
\begin{theorem}[\cite{Ch2}]\label{ehmfinite}
 Let $g=e^{2\varphi}|dz|^2$ be an \ehm~in a punctured disk $D \setminus \{0 \}$ with finite
area and finite Calabi energy,
 then \BE
 \lim_{z\rightarrow 0} |z|^2 \cdot K \cdot e^{2\varphi}=0.
\EE
\end{theorem}

Further X.X.Chen in \cite{Ch2} proved the following theorem:
\begin{theorem}[\cite{Ch2}]\label{ehmcuspth}
 Let $g=e^{2\varphi}|dz|^2$ be an \ehm~on a punctured disk $D \setminus \{0\}$ with finite area and
finite Calabi energy and suppose $g$ has
 a weak cusp singular point at $z=0$. Then the following statements hold
 \begin{itemize}
 \item[(1)] There exists a constant $C_1$ such that
  \BE
   |K_{,zz}|\cdot |z| \leq C_1.
  \EE
 \item[(2)] There exists a constant $C_2$ such that
  \BE\
   |K_{,z}| \leq C_2 \cdot |z| \cdot e^{2\varphi}.
  \EE
 \item[(3)] There exists a negative constant $C_3$ such that
\[\dps{\lim_{z\rightarrow 0}}K=C_3.
\]
 \end{itemize}

\end{theorem}

\section{Proof of the main theorem}
 \subsection{Proof of the necessity of the main theorem}
 Let $\Sigma^* = \S*$. Since $g$ is an HCMU metric on $\Sigma^*$, then on
 $\Sigma^*$
 \[K_{,zz}=0,
 \]
 which is equivalent to the fact that
 \[ \nabla K \triangleq \sqrt{-1}K^{,z}\frac{\partial}{\partial z}
 \]
 is a holomorphic vector field on $\Sigma^*$. Since an HCMU metric is an \ehm, on $\Sigma^*$
 \BE\label{ehmequ3}
  \Delta_{g}K+K^2=C.
 \EE
We will first prove the following proposition:
\begin{proposition}\label{gradeprop}
 There exists a real constant $C'$ such that
\BE\label{gradequ}
 -4 \sqrt{-1}\nabla K(K)=-\frac{K^3}{3}+CK+C'~~~\text{on}~~\Sigma^*.
\EE
\end{proposition}
\begin{proof}
 Since $g$ is not a CSC metric, there exists $p \in \Sigma^*$ such that $dK(p)\neq 0$. Let $(U,z)$ be
 a local complex coordinate chart around $p$ such that $U$ is connected and $dK$ does not vanish
 on $U$. Suppose $g=e^{2\varphi}|dz|^2$ on $U$. Then on $U$
 \[ \nabla K=\sqrt{-1}K^{,z}\frac{\partial}{\partial z}=\sqrt{-1}e^{-2\varphi}K_{\bar{z}}\frac{\partial}{\partial z}
 \]
 and
 \[ 4 K_{z\bar{z}}=(C-K^2)e^{2\varphi}.
 \]
 Let $F=4e^{-2\varphi}K_{\bar{z}}$, then $F$ is a holomorphic function on $U$ and does not
vanish on $U$. Therefore
 \[4 K_{z\bar{z}}=(C-K^2)e^{2\varphi}=(C-K^2)\frac{4K_{\bar{z}}}{F},
 \]
so
 \[ K_{z\bar{z}}=(\frac{-\frac{K^3}{3}+CK}{F})_{\bar{z}},
 \]
 which means
  $\dps{F_1\triangleq K_{z}-\frac{-\frac{K^3}{3}+CK}{F}}$ is a holomorphic function on $U$. Then
 \[ FK_{z}=-\frac{K^3}{3}+CK+FF_1.
 \]
Since $FK_z=4e^{-2\varphi}K_{\bar{z}}K_z=4e^{-2\varphi}|K_z|^2$ is a real function, $FF_1$ which is a
 holomorphic function on $U$ is a real constant. We denote it by $C'_U$, so we have
 \[ -4\sqrt{-1}\nabla K(K)=FK_z=-\frac{K^3}{3}+CK+C'_U~~\text{on}~U.
 \]
 \par
 Next let $\mathfrak{S}=\{p \in \Sigma^*| dK(p)=0\}$, then since $g$ is not a CSC metric $\mathfrak{S}$
 is a discrete set of $\Sigma^*$. Pick any point $q \in \mathfrak{S}$ and let $(V,w)$ be a local complex
 coordinate chart around $q$ such that $w(V)$ is a disk and $dK$ does not vanish on $V \setminus \{q\}$.
 Suppose $g=e^{2\psi}|dw|^2$ on $V$. Then $G\triangleq 4e^{-2\psi}K_{\bar{w}}$ is a holomorphic function
 on $V$ and does not vanish on $V \setminus \{q\}$. Similar to above, there is a real constant $C'_V$ such
 that
 \BE\label{pregradient}
  G K_{w}=-\frac{K^3}{3}+CK+C'_V~~\text{on}~V \setminus \{q\}.
 \EE
Since both sides of (\ref{pregradient}) are continuous at $q$, we have
\[ -4\sqrt{-1}\nabla K(K)=-\frac{K^3}{3}+CK+C'_V~~\text{on}~V.
\]
\par
Consequently $-4\sqrt{-1}\nabla K(K)+\cfrac{K^3}{3}-CK$ is a locally constant function on $\Sigma^*$.
Since $\Sigma^*$ is connected, $-4\sqrt{-1}\nabla K(K)+\cfrac{K^3}{3}-CK$
is a global constant on $\Sigma^*$. Therefore there is a real constant $C'$ such that
\[-4 \sqrt{-1}\nabla K(K)=-\frac{K^3}{3}+CK+C'~~~\text{on}~~\Sigma^*.\]
We prove the proposition.
\end{proof}
Define $\mathfrak{S}=\{p \in \Sigma^*| dK(p)=0\}$. Then we have the following proposition:
\begin{proposition}\label{finiteprop}
 $\mathfrak{S}$ is finite.
\end{proposition}
\begin{proof}
 If $\mathfrak{S}$ is infinite, then $\mathfrak{S}$ has cluster points in $\Sigma$
since $\Sigma$ is compact.  Suppose $ e^* \in \Sigma$ is one of the cluster points
of $\mathfrak{S}$.  Obviously $e^* \notin \Sigma^*$, so $e^*=p_i$, some $i$,
$i \in \{1,2,\ldots,I\}$ or $e^*=q_j$, some $j$, $j \in \{1,2,\ldots,J\}$.
 \begin{itemize}
 \item[Case 1:] $e^*=p_i$, some $i$, $i \in \{1,2,\ldots,I\}$. \\
  Let $(U,z)$ be a local complex coordinate chart around $p_i$ such that $U \setminus
  \{p_i\} \subset \Sigma^*$, $z(U)$ is a disk $D$ and
  $z(p_i)=0$. Suppose $g=e^{2\varphi}|dz|^2$ on $U \setminus \{p_i\}$.
  Let $F=4 e^{-2\varphi}K_{\bar{z}}$, then by (2) in Theorem \ref{ehmcuspth} $F$ is actually a
   holomorphic function on $D$ and has a zero at $0$.
  Since $p_i$ is a cluster point of $\mathfrak{S}$, $0$ is a cluster point of the zeros of $F$.
Then $F\equiv 0$ on $D$, which means $\nabla K\equiv 0$ on $\Sigma^*$. It is impossible.

 \item[Case 2:] $e^*=q_j$, some $j$, $j \in \{1,2,\ldots,J\}$. \\
 Let $(V,w)$ be a local complex coordinate chart around $q_j$ such that $V \setminus \{q_j\}
\subset \Sigma^*$, $w(V)$ is a disk $D'$ and $w(q_j)=0$.
Suppose $g=e^{2\psi}|dw|^2$ on $V \setminus \{q_j\}$. Let $G=4 e^{-2\psi}K_{\bar{w}}$, then
$G$ is a holomorphic
 function on $D' \setminus \{0\}$. Since $q_j$ is a cluster point of $\mathfrak{S}$,
$0$ is a cluster point of the zeros of $G$. Then $0$ is a zero of $G$ or an essential
singularity of $G$. If $0$ is a zero of $G$, we can  get $G\equiv 0$ on $D'$.
It is impossible. Therefore $0$ is an essential singularity of $G$.
 By Theorem \ref{ehmfinite}
 \[ \lim_{w\rightarrow 0} |w|^2 \cdot K \cdot e^{2\psi} =0.
 \]
 On the other hand, since $q_j$ is a conical singularity of $g$ with the angle $2 \pi \alpha_j$,
 by (\ref{coneequ2}) there exists a positive continuous function $h$ on $D'$ such that
 \[ e^{2\psi}=\frac{h}{|w|^{2-2\alpha_j}}~~\text{on}~D'\setminus \{0\}.
 \]
 Therefore we have
 \BE\label{orderK}
  \lim_{w \rightarrow 0} K \cdot |w|^{2\alpha_j}=0.
 \EE
 By Proposition \ref{gradeprop}, (\ref{gradequ}) holds on $V \setminus \{q_j\}$, that is,
 \[ \frac{1}{4} |G|^2 \cdot h\cdot |w|^{2\alpha_j-2}=-\frac{K^3}{3}+CK+C'~~\text{on}~D'\setminus \{0\}.
 \]
 Then by (\ref{orderK}), there exists $b \in \mathds{N}$ such that
 \[ \lim_{w \rightarrow 0} |G|^2 \cdot |w|^{2b} =0,
 \]
 which means
 \[ \lim_{w \rightarrow 0} |G \cdot w^b|=0.
 \]
 It is a contradiction since $0$ is an essential singularity of $G$.
 \end{itemize}
 Consequently we exclude Case 1 and Case 2. That means $\mathfrak{S}$ is finite. We prove the proposition.
 \end{proof}
 Now let $\mathfrak{S}=\{ e_1,e_2,\cdots, e_S\}$ and $\Sigma' = \Sigma^* \setminus \mathfrak{S}$,
 then $\nabla K$ has no zeros on $\Sigma'$. We define the dual 1-form of $\nabla K$
 as the {\bf Character 1-form} of $g$.
 Locally let $(\mathfrak{U},\mathfrak{z})$ be a local complex coordinate chart on $\Sigma'$ and suppose $g=
 e^{2u}|d\mathfrak{z}|^2$ on $\mathfrak{U}$, then
 \[\nabla K=\sqrt{-1}e^{-2u}K_{\bar{\mathfrak{z}}}
\cfrac{\partial}{\partial \mathfrak{z}}=\sqrt{-1}\frac{\mathcal{F}}{4} \cfrac{\partial}{\partial \mathfrak{z}}
 ~~\text{on}~\mathfrak{U}.
 \]
 Define the {\bf Character 1-form} $\omega=\cfrac{d\mathfrak{z}}{\mathcal{F}}$ on $\mathfrak{U}$.
 Then we have the following proposition:
 \begin{proposition}\label{propertiesofomega}
  $\omega$ has the following properties:
  \begin{itemize}
   \item[(1)] $\omega$ is a meromorphic 1-form on $\Sigma$.
   \item[(2)] On $\Sigma'$,
   \BE\label{partialKequ}
   \partial K=(-\frac{K^3}{3}+CK+C')\omega
   \EE
   or equivalently
   \BE\label{dKequ}
    dK=(-\frac{K^3}{3}+CK+C')(\omega+\bar{\omega}).
   \EE
   \item[(3)] On $\Sigma'$,
   \BE\label{gequ}
    g=4(-\frac{K^3}{3}+CK+C')\omega\bar{\omega}.
   \EE
  \end{itemize}
 \end{proposition}
 \begin{proof}
  \begin{itemize}
   \item[(1):] Obviously $\omega$ is holomorphic on $\Sigma'$. By (2) in Theorem \ref{ehmcuspth},
   each $p_i, i=1,2,\cdots,I,$ is a pole of $\omega$. Since each $e_s, s=1,2,\cdots,S,$ is a zero of
   $\nabla K$, each $e_s, s=1,2,\cdots, S,$ is a pole of $\omega$. Pick any $q_j,j=1,2,\ldots,J$, and let
   $(U,z)$ be a local complex coordinate chart around $q_j$ such that $U\setminus \{q_j\} \subset
  \Sigma'$, $z(U)$ is a disk $D$ and $z(q_j)=0$.
   Suppose $g=e^{2\varphi}|dz|^2$ on $U \setminus \{q_j\}$. Then $F=4 e^{-2\varphi}K_{\bar{z}}$ is a holomorphic
   function on $D \setminus \{0\}$, so $0$ is a removable singularity or a pole or an essential singularity
   of $F$. Then we can use the same argument in Case 2 in the proof of Proposition \ref{finiteprop}
   to prove $0$ is not an essential singularity of $F$. Hence $0$ is a removable singularity or a pole of
   $F$, which shows $q_j$ is a regular point or a pole of $\omega$(note $\omega=\cfrac{dz}{F}$ on $U
   \setminus \{q_j\}$). Then we finish the proof of (1).
   \item[(2), (3):] Pick any point $p \in \Sigma'$ and let $(V,w)$ be a \lccc around $p$ such that $w(V)$ is
   a disk $D'$ and $w(p)=0$.  Suppose $g=e^{2\psi}|dw|^2$ on $V$. Then $G=4e^{-2\psi}
  K_{\bar{w}}$ is a nonvanishing holomorphic function on $V$. Since (\ref{gradequ}) holds on $V$, we have
   \[ GK_w=-\frac{K^3}{3}+CK+C',
   \]
   that is,
   \[ K_w dw=(-\frac{K^3}{3}+CK+C')\frac{dw}{G},
   \]
   which is
   \[ \partial K|_V=(-\frac{K^3}{3}+CK+C')\omega|_V,
   \]
   so we prove (2). \\
   On the other hand, $G=4e^{-2\psi}K_{\bar{w}}$, that is, $e^{2\psi}=\cfrac{4K_{\bar{w}}}{G}$, which means
   \[ g|_V=e^{2\psi}dw d\bar{w}=\frac{dw}{G}4K_{\bar{w}}d\bar{w}=\omega|_V 4\bar{\partial}K|_V=4
(-\frac{K^3}{3}+CK+C')
   \omega|_V \bar{\omega}|_V.
   \]
   Then we prove (3).
  \end{itemize}
 \end{proof}
 In the following we will study the roots of $-\cfrac{K^3}{3}+CK+C'=0$, the canonical divisor of $\omega$ and
the residues of $\omega$. \par
First by (3) in Theorem \ref{ehmcuspth}, $K$ is continuous at each $p_i, i=1,2,\cdots, I$, and
 there are $I$ negative numbers $b_1,b_2,\cdots, b_I$ such that $\dps{\lim_{p \rightarrow p_i}}K(p)=b_i, i=1,2,\cdots,I$.
 Then we have the following lemma:
 \begin{lemma}\label{cusplemma}
  Each $b_i, i=1,2,\cdots, I$, is a root of $-\cfrac{K^3}{3}+CK+C'=0$ .
 \end{lemma}
 \begin{proof}
  Pick any $p_i$ and let $(U,z)$ be a \lccc around $p_i$ such that $U \setminus \{p_i\} \subset \Sigma'$,
  $z(U)$ is a disk $D$ and $z(p_i)=0$.
  Suppose $g=e^{2\varphi}|dz|^2$ on $U \setminus \{p_i\}$. Then $F=4e^{-2\varphi}K_{\bar{z}}$ is actually a
  holomorphic function on $D$ and $0$ is a zero of $F$ on $D$. Next on $U \setminus \{p_i\}$ (\ref{gradequ}) holds,
  that is,
  \[ FK_z=\frac{1}{4}|F|^2 e^{2\varphi}=-\frac{K^3}{3}+CK+C'.
  \]
  Since $0$ is a zero of $F$, we assume $F=z\widetilde{F}$ on $D$, where $\widetilde{F}$ is a holomorphic
  function on $D$. Then on $D \setminus \{0\}$
  \BE\label{cuspequ1}
   \frac{1}{4} |\widetilde{F}|^2 |z|^2 e^{2\varphi}=\frac{1}{4}|\widetilde{F}|^2 e^{2(\varphi+\ln r)}=-\frac{K^3}{3}+CK+C',
  \EE
  where $r=|z|$. It follows from (1) in Theorem \ref{finiteth} that
  \[ 0=-\frac{b_i^3}{3}+Cb_i+C'.
  \]
  Then we prove the lemma.
 \end{proof}
 By Lemma \ref{cusplemma}, we get $-\cfrac{K^3}{3}+CK+C'=0$ has not a triple root. Otherwise
 \[ -\frac{K^3}{3}+CK+C'=-\frac{1}{3}(K-a)^3,
 \]
 then $a=0$, but by Lemma \ref{cusplemma}, each $b_i, i=1,2,\cdots,I, $ which is negative is a root
 of $-\cfrac{K^3}{3}+CK+C'=0$. It is a contradiction. Therefore there are four cases to consider:
 \begin{itemize}
  \item[(C1):] $-\cfrac{K^3}{3}+CK+C'=-\cfrac{1}{3}(K-2\mu)[(K+\mu)^2+\mu^*]$, where
  $\mu<0,~\mu^*>0$.
  \item[(C2):] $-\cfrac{K^3}{3}+CK+C'=-\cfrac{1}{3}(K-2\mu)(K+\mu)^2$, where $\mu<0$.
  \item[(C3):] $-\cfrac{K^3}{3}+CK+C'=-\cfrac{1}{3}(K-\mu)(K-\mu^*)(K+\mu+\mu^*)$, where
  $\mu \neq \mu^* ,~\mu\neq -(\mu+\mu^*),~\mu^*> -(\mu+\mu^*)$ and there
  exists some $b_i$ such that $b_i=\mu$.
  \item[(C4):] $-\cfrac{K^3}{3}+CK+C'=-\cfrac{1}{3}(K-\mu)^2(K+2\mu)$, where $\mu<0$.
 \end{itemize}
 Next we will exclude (C1), (C2) and (C3). First we exclude (C1):
 \begin{lemma}\label{C1lem}
 (C1) can not hold.
 \end{lemma}
 \begin{proof}
Suppose that (C1) holds. Pick any $p_i$ and let $(U,z)$ be a \lccc around $p_i$ such that
$U \setminus \{p_i\} \subset \Sigma'$, $z(U)$ is a disk $D$ and $z(p_i)=0$.
  Suppose $g=e^{2\varphi}|dz|^2$ on $U \setminus \{p_i\}$.
  Then $F=4e^{-2\varphi}K_{\bar{z}}$ is a holomorphic function on $D$, $0$ is a unique zero
  of $F$ on $D$ and $\omega=\cfrac{dz}{F}$ on $U \setminus \{p_i\}$. Suppose $\cfrac{1}{F}$
  has the following expression on $D \setminus \{0\}$:
  \BE\label{laurentexpr1}
   \frac{1}{F}=\frac{\lambda_{-k}}{z^k}+\cdots+\frac{\lambda_{-2}}{z^2}+\frac{\lambda_{-1}}{z}+
   \sum_{m=0}^\infty \lambda_m z^m=\frac{\Phi(z)}{z^k},
  \EE
  where $\Phi(z)$ is a holomorphic function on $D$ with $\Phi(0)=\lambda_{-k}\neq 0$. Then
  \BE\label{omegaexpr1}
   \omega=\frac{dz}{F}=\frac{\lambda_{-1}}{z}dz+df_1,
  \EE
  where $f_1=\cfrac{f_2}{z^{k-1}}$ and $f_2$ is a holomorphic function on $D$ with $f_2(0) \neq 0$.
  By Proposition \ref{propertiesofomega}, on $U \setminus \{p_i\}$ (\ref{dKequ}) holds. Then we
  substitute $\omega=\cfrac{\lambda_{-1}}{z}dz+df_1$ and $-\cfrac{K^3}{3}+CK+C'=-\cfrac{1}{3}(K-2\mu)[(K+\mu)^2+
  \mu^*]$ into (\ref{dKequ}) to get on $U \setminus \{p_i\}$
  \BE\label{dKequC1}
   (-3) \frac{dK}{(K-2\mu)[(K+\mu)^2+\mu^*]}=\frac{\lambda_{-1}dz}{z}+\frac{\overline{\lambda_{-1}}}{\bar{z}}
   d\bar{z}+d(2Re(f_1)).
  \EE
   Suppose
  \[ \frac{1}{(K-2\mu)[(K+\mu)^2+\mu^*]}=\frac{\beta_1}{K-2\mu}+\frac{\beta_2(K+\mu)+\beta_3}
  {(K+\mu)^2+\mu^*}.
  \]
  Then
 \begin{eqnarray}\label{IKC1}
  \nonumber  \frac{dK}{(K-2\mu)[(K+\mu)^2+\mu^*]}= \\
   d\{\beta_1 \ln(2\mu-K)+\frac{\beta_2}{2}\ln [(K+\mu)^2+\mu^*]+\frac{\beta_3}{\sqrt{\mu^*}}
   \arctan \frac{K+\mu}{\sqrt{\mu^*}}\},
 \end{eqnarray}
 where we use the fact that on $U \setminus \{p_i\}$ (\ref{gequ}) holds, so
 \[ -\cfrac{K^3}{3}+CK+C'=-\cfrac{1}{3}(K-2\mu)[(K+\mu)^2+\mu^*]>0~~\text{on}~~U \setminus \{p_i\},
 \]
 which implies $K <2\mu$ on $U \setminus \{p_i\}$. By (\ref{dKequC1}) and (\ref{IKC1}), we have $\lambda_{-1} \in
 \mathds{R}$.
 Then we can integrate both sides of (\ref{dKequC1}) to get on $D \setminus \{0 \}$
 \begin{eqnarray}\label{IKequC1}
  \nonumber (-3)\{\beta_1 \ln(2\mu-K)+\frac{\beta_2}{2}\ln [(K+\mu)^2+\mu^*]+\frac{\beta_3}{\sqrt{\mu^*}}
   \arctan \frac{K+\mu}{\sqrt{\mu^*}}\}= \\
   \lambda_{-1} \ln |z|^2+2Re(f_1)+c,
 \end{eqnarray}
 where $c$ is a real constant. \par
 Next since on $U \setminus \{p_i\}$ (\ref{gequ}) holds,
 \begin{eqnarray*}
   e^{2\varphi}=-\frac{4}{3}(K-2\mu)[(K+\mu)^2+\mu^*]\frac{|\Phi(z)|^2}{|z|^{2k}},
 \end{eqnarray*}
 that is,
 \[ \varphi=\frac{1}{2}(\ln \frac{2\mu-K}{|z|^{2k}}+\ln \frac{4[(K+\mu)^2+\mu^*]|\Phi(z)|^2}{3}).
 \]
 Then $\dps{\lim_{z \rightarrow 0}}\cfrac{\varphi+\ln |z|}{\ln |z|}=0$ implies
 \BE\label{cusplimC1t}
  \lim_{z\rightarrow 0}\frac{\ln (2\mu-K)}{\ln |z|}=2(k-1).
 \EE
 If $k=1$, then $\dps{\lim_{z\rightarrow 0}}\frac{\ln (2\mu-K)}{\ln |z|}=0$. Divide both sides of
 (\ref{IKequC1}) by $\ln |z|$, let $z \rightarrow 0$ and take limits. We get
 \BE\label{cusplimC1a}
  \lim_{z \rightarrow 0}\frac{(-\frac{3}{2}\beta_1)\ln (2\mu-K)}{\ln |z|}=\lambda_{-1},
 \EE
 so $\lambda_{-1}=0$. It is impossible. If $k>1$, then $\dps{\lim_{z\rightarrow 0}}\frac{\ln (2\mu-K)}{\ln |z|}
 =2(k-1)>0$. We rewrite (\ref{IKequC1}) to be
 \begin{eqnarray}\label{IKequC1'}
  \nonumber (-3)\{\beta_1 \ln(2\mu-K)+\frac{\beta_2}{2}\ln [(K+\mu)^2+\mu^*]+\frac{\beta_3}{\sqrt{\mu^*}}
   \arctan \frac{K+\mu}{\sqrt{\mu^*}}\}= \\
   \lambda_{-1} \ln |z|^2+\frac{f_2}{z^{k-1}}+\frac{\overline{f_2}}{\bar{z}^{k-1}}+c.
 \end{eqnarray}
 Multiply both sides of (\ref{IKequC1'}) by $z^{k-1}$ to get
 \begin{eqnarray}\label{IKequC1''}
 \nonumber (-3)\{\beta_1 z^{k-1}\ln(2\mu-K)+\frac{\beta_2}{2}z^{k-1}\ln [(K+\mu)^2+\mu^*]+\frac{\beta_3}{\sqrt{\mu^*}}
 z^{k-1} \arctan \frac{K+\mu}{\sqrt{\mu^*}}\}= \\
 \lambda_{-1} z^{k-1}\ln |z|^2+f_2+\overline{f_2}\frac{z^{k-1}}{\bar{z}^{k-1}}+cz^{k-1}.
 \end{eqnarray}
Let $z \rightarrow 0$ on both sides of (\ref{IKequC1''}) and take limits.
 Note
 \[ \lim_{z\rightarrow 0} z^{k-1}\ln (2\mu-K)=\lim_{z \rightarrow 0} \frac{\ln (2\mu-K)}{\ln |z|} z^{k-1}
 \ln |z|=0.
 \]
  Therefore we get $\dps{\lim_{z\rightarrow 0}
 \frac{z^{k-1}}{\bar{z}^{k-1}}}=A,~ A\neq 0$. It is impossible. Consequently we exclude (C1).
\end{proof}

   Then we exclude (C2):
\begin{lemma}\label{C2lem}
 (C2) can not hold.
\end{lemma}
\begin{proof}
 Suppose that (C2) holds. Pick any $p_i$ and let $(U,z)$ be a \lccc around $p_i$ such that $U \setminus
 \{p_i\} \subset \Sigma'$, $z(U)$ is a disk $D$ and $z(p_i)=0$. Suppose $g=e^{2\varphi}|dz|^2$ on $U \setminus \{p_i\}$.
  Then $F=4e^{-2\varphi}K_{\bar{z}}$ is a holomorphic function on $D$, $0$ is a unique zero
  of $F$ on $D$ and $\omega=\cfrac{dz}{F}$ on $U \setminus \{p_i\}$. Suppose $\cfrac{1}{F}$
  has the following expression on $D \setminus \{0\}$:
  \BE\label{laurentexpr2}
   \frac{1}{F}=\frac{\lambda_{-k}}{z^k}+\cdots+\frac{\lambda_{-2}}{z^2}+\frac{\lambda_{-1}}{z}+
   \sum_{m=0}^\infty \lambda_m z^m=\frac{\Phi(z)}{z^k},
  \EE
  where $\Phi(z)$ is a holomorphic function on $D$ with $\Phi(0)=\lambda_{-k}\neq 0$. Then
  \BE\label{omegaexpr2}
   \omega=\frac{dz}{F}=\frac{\lambda_{-1}}{z}dz+df_1,
  \EE
  where $f_1=\cfrac{f_2}{z^{k-1}}$ and $f_2$ is a holomorphic function on $D$ with $f_2(0) \neq 0$.
  Similar to the proof of Lemma \ref{C1lem}, we
  substitute $\omega=\cfrac{\lambda_{-1}}{z}dz+df_1$ and $-\cfrac{K^3}{3}+CK+C'=-\cfrac{1}{3}(K-2\mu)
  (K+\mu)^2$ into (\ref{dKequ}) to get on $U \setminus \{p_i\}$
  \BE\label{dKequC2}
   (-3)\frac{dK}{(K-2\mu)(K+\mu)^2}=\frac{\lambda_{-1}dz}{z}+\frac{\overline{\lambda_{-1}}}{\bar{z}}
   d\bar{z}+d(2Re(f_1)).
  \EE
   Suppose
  \[ \frac{1}{(K-2\mu)(K+\mu)^2}=\frac{\beta_1}{K-2\mu}+\frac{\beta_2}{K+\mu}+
  \frac{\beta_3}{(K+\mu)^2}.
  \]
  Then
  \begin{eqnarray}\label{IKC2}
  \nonumber  \frac{dK}{(K-2\mu)(K+\mu)^2}= \\
   d[\beta_1 \ln (2\mu-K)+\beta_2 \ln (-\mu-K)-\frac{\beta_3}{K+\mu}],
 \end{eqnarray}
 where we use the fact that on $U \setminus \{p_i\}$ (\ref{gequ}) holds, so
 \[ -\cfrac{K^3}{3}+CK+C'=-\cfrac{1}{3}(K-2\mu)(K+\mu)^2>0~~\text{on}~~U \setminus \{p_i\},
 \]
 that is, $K <2\mu<-\mu$. By (\ref{dKequC2}) and (\ref{IKC2}), we have $\lambda_{-1} \in \mathds{R}$.
 Integrate both sides of (\ref{dKequC2}) to get on $D \setminus \{0\}$
 \begin{eqnarray}\label{IKequC2}
   \nonumber (-3)[\beta_1\ln (2\mu-K)+\beta_2 \ln (-\mu-K)-\frac{\beta_3}{K+\mu}]= \\
   \lambda_{-1} \ln |z|^2+2Re(f_1)+c,
 \end{eqnarray}
 where $c$ is a real constant. \par
 On the other hand, since on $U \setminus \{p_i\}$ (\ref{gequ})
 holds,
 \[ e^{2\varphi}=-\frac{4}{3}(K-2\mu)(K+\mu)^2 \frac{|\Phi|^2}{|z|^{2k}},
 \]
 that is,
 \[ \varphi=\frac{1}{2}(\ln \frac{2\mu-K}{|z|^{2k}}+\ln \frac{4(K+\mu)^2 |\Phi|^2}{3}).
 \]
 Then $\dps{\lim_{z \rightarrow 0}}\cfrac{\varphi+\ln |z|}{\ln |z|}=0$ implies
 \BE\label{cusplimC2t}
  \lim_{z\rightarrow 0}\frac{\ln (2\mu-K)}{\ln |z|}=2(k-1).
 \EE
 Finally we use the same method in the proof of Lemma \ref{C1lem} to get whether $k=1$ or $k>1$, there
 exists a contradiction. Hence we exclude (C2).
 \end{proof}
Finally we exclude (C3):
 \begin{lemma}\label{C3lem}
 (C3) can not hold.
 \end{lemma}
 \begin{proof}
 Suppose (C3) holds. Fix a point $p_i$ which satisfies $b_i=\mu$. Let $(U,z)$ be a \lccc around $p_i$ such that
 $U \setminus \{p_i\} \subset \Sigma'$, $z(U)$ is a disk $D$ and $z(p_i)=0$.
 Suppose $g=e^{2\varphi}|dz|^2$ on $U \setminus \{p_i\}$.
  Then $F=4e^{-2\varphi}K_{\bar{z}}$ is a holomorphic function on $D$, $0$ is a unique zero
  of $F$ on $D$ and $\omega=\cfrac{dz}{F}$ on $U \setminus \{p_i\}$. Suppose $\cfrac{1}{F}$
  has the following expression on $D \setminus \{0\}$:
 \BE\label{laurentexpr3}
   \frac{1}{F}=\frac{\lambda_{-k}}{z^k}+\cdots+\frac{\lambda_{-2}}{z^2}+\frac{\lambda_{-1}}{z}+
   \sum_{m=0}^\infty \lambda_m z^m=\frac{\Phi(z)}{z^k},
  \EE
  where $\Phi(z)$ is a holomorphic function on $D$ with $\Phi(0)=\lambda_{-k}\neq 0$. Then
  \BE\label{omegaexpr3}
   \omega=\frac{dz}{F}=\frac{\lambda_{-1}}{z}dz+df_1,
  \EE
   where $f_1=\cfrac{f_2}{z^{k-1}}$ and $f_2$ is a holomorphic function on $D$ with $f_2(0) \neq 0$.
  Then we substitute $\omega=\cfrac{\lambda_{-1}}{z}dz+df_1$ and $-\cfrac{K^3}{3}+CK+C'=-\cfrac{1}{3}(K-\mu)
  (K-\mu^*)(K+\mu+\mu^*)$
  into (\ref{dKequ}) to get on $U \setminus \{p_i\}$
  \BE\label{dKequC3}
   (-3) \frac{dK}{(K-\mu)(K-\mu^*)(K+\mu+\mu^*)}=\frac{\lambda_{-1}dz}{z}+
   \frac{\overline{\lambda_{-1}}}{\bar{z}}
   d\bar{z}+d(2Re(f_1)).
  \EE
   Suppose
  \[ \frac{1}{(K-\mu)(K-\mu^*)(K+\mu+\mu^*)}=\frac{\beta_1}{K-\mu}+\frac{\beta_2}{K-\mu^*}
  +\frac{\beta_3}{K+\mu+\mu^*}.
  \]
  Then
  \begin{eqnarray}\label{IKC3}
  \nonumber  \frac{dK}{(K-\mu)(K-\mu^*)(K+\mu+\mu^*)}= \\
   d(\beta_1 \ln |K-\mu|+\beta_2 \ln |K-\mu^*|+\beta_3 \ln |K+\mu+\mu^*|).
 \end{eqnarray}
 Similar to above, $\lambda_{-1} \in \mathds{R}$. Integrate both sides of
 (\ref{dKequC3}) to get on $D \setminus \{0\}$
 \begin{eqnarray}\label{IKequC3}
   \nonumber (-3)(\beta_1 \ln |K-\mu|+\beta_2 \ln |K-\mu^*|+\beta_3 \ln |K+\mu+\mu^*|)= \\
   \lambda_{-1} \ln |z|^2+2Re(f_1)+c,
 \end{eqnarray}
 where $c$ is a real constant. \par
 On the other hand, since on $U \setminus \{p_i\}$ (\ref{gequ}) holds,
 \[ e^{2\varphi}=\frac{4}{3}|K-\mu||K-\mu^*||K+\mu+\mu^*|\frac{|\Phi|^2}{|z|^{2k}},
  \]
  that is,
 \[ \varphi=\frac{1}{2}(\ln \frac{|K-\mu|}{|z|^{2k}}+\ln \frac{4|K-\mu^*||K+\mu
  +\mu^*||\Phi|^2}{3}).
  \]
  Then $\dps{\lim_{z \rightarrow 0}}\cfrac{\varphi+\ln |z|}{\ln |z|}=0$ implies
 \BE\label{cusplimC3t}
  \lim_{z\rightarrow 0}\frac{\ln |K-\mu|}{\ln |z|}=2(k-1).
 \EE
 We can also use the same method in the proof of Lemma \ref{C1lem} to exclude (C3).
 \end{proof}
 By Lemma \ref{cusplemma}, Lemma \ref{C1lem}, Lemma \ref{C2lem} and Lemma \ref{C3lem},
 we obtain the following theorem:
  \begin{theorem}\label{cuspth1}
   There exists a negative number $\mu$ such that $b_1=b_2=\cdots=b_I=\mu$ and
 \[ -\frac{K^3}{3}+CK+C'=-\frac{1}{3}(K-\mu)^2(K+2\mu).
 \]
 \end{theorem}
 \begin{proof}
 Since we have excluded cases (C1), (C2) and (C3), (C4) holds, that is,
 $-\cfrac{K^3}{3}+CK+C'=-\cfrac{1}{3}(K-\mu)^2(K+2\mu)$,
 where $\mu<0$. By Lemma \ref{cusplemma} each $b_i$, $i=1,2,\cdots,I$, is a root of
  $-\cfrac{K^3}{3}+CK+C'=0$. Since (C4) holds, $-\cfrac{K^3}{3}+CK+C'=0$ has
  a unique negative root $\mu$. Hence
 $b_1=b_2=\cdots=b_I=\mu$. We prove the theorem.
 \end{proof}
 By Theorem \ref{cuspth1}, we get on $\Sigma'$ $K< -2\mu$. Then there are two possibilities:
 one is $\forall p \in \Sigma'$, $K(p)<\mu$; the other is
 $\forall p \in \Sigma'$, $\mu<K(p)<-2\mu$. In the following
 we will get $\forall p \in \Sigma', ~\mu<K(p)<-2\mu$. Before the result, we will give
 the following three lemmas.
 \begin{lemma}\label{boundlemma1}
  If $\forall p \in \Sigma',~K(p)<\mu$, then each  $p_i,~i=1,2,\cdots,I$, is a simple pole of
 $\omega$ and the residue of $\omega$ at each $p_i$, $i=1,2,\cdots,I,$ is a negative real number.
 \end{lemma}
 \begin{proof}
    Pick any $p_i$ and let $(U,z)$ be a \lccc around $p_i$ such that $U \setminus \{p_i\} \subset \Sigma'$,
  $z(U)$ is a disk $D$ and $z(p_i)=0$. Suppose $g=e^{2\varphi}|dz|^2$ on $U \setminus \{p_i\}$.
  Then $F=4e^{-2\varphi}K_{\bar{z}}$ is a holomorphic function on $D$, $0$ is a unique zero
  of $F$ on $D$ and $\omega=\cfrac{dz}{F}$ on $U \setminus \{p_i\}$. Suppose $\cfrac{1}{F}$
  has the following expression on $D \setminus \{0\}$:
  \BE\label{laurentexpr4}
   \frac{1}{F}=\frac{\lambda_{-k}}{z^k}+\cdots+\frac{\lambda_{-2}}{z^2}+\frac{\lambda_{-1}}{z}+
   \sum_{m=0}^\infty \lambda_m z^m=\frac{\Phi(z)}{z^k},
  \EE
  where $\Phi(z)$ is a holomorphic function on $D$ with $\Phi(0)=\lambda_{-k}\neq 0$. Then
  \BE\label{omegaexpr4}
   \omega=\frac{dz}{F}=\frac{\lambda_{-1}}{z}dz+df_1,
  \EE
  where $f_1=\cfrac{f_2}{z^{k-1}}$ and $f_2$ is a holomorphic function on $D$ with $f_2(0) \neq 0$.
  Then we substitute $\omega=\cfrac{\lambda_{-1}}{z}dz+df_1$ and $-\cfrac{K^3}{3}+CK+C'=-\cfrac{1}{3}(K-\mu)^2
  (K+2\mu)$ into (\ref{dKequ}) to get on $U \setminus \{p_i\}$
  \BE\label{dKequcusp}
   (-3) \frac{dK}{(K-\mu)^2(K+2\mu)}=\frac{\lambda_{-1}dz}{z}+
   \frac{\overline{\lambda_{-1}}}{\bar{z}}
   d\bar{z}+d(2Re(f_1)).
  \EE
   Suppose
  $$ \frac{1}{(K-\mu)^2(K+2\mu)}=\frac{\beta_1}{K+2\mu}+\frac{\beta_2}{K-\mu}+\frac{\beta_3}
  {(K-\mu)^2}. $$
 Then $\beta_1=\cfrac{1}{9\mu^2},~\beta_2=-\cfrac{1}{9\mu^2},~\beta_3=\cfrac{1}{3\mu}$ and
 on $U \setminus \{p_i\}$
\begin{equation}\label{dKequ'}
  \frac{dK}{(K-\mu)^2 (K+2\mu)}=\beta_1d[\ln (-2\mu-K)-\ln (\mu-K)-\frac{3\mu}{K-\mu}].
\end{equation}
 We can also get $\lambda_{-1} \in \mathds{R}$. Integrate both
 sides of (\ref{dKequcusp}) to get on $D \setminus \{0\}$,
 \bey\label{IKequ'}
 \nn (-3\beta_1)[\ln (-2\mu-K)-\ln (\mu-K)-\frac{3\mu}{K-\mu}]= \\
 \lambda_{-1} \ln |z|^2+2Re(f_1)+c,
 \eey
 where $c$ is a real constant. \par
 On the other hand,
 \[ e^{2\varphi}=-\frac{4}{3}(K-\mu)^2(K+2\mu)\frac{|\Phi|^2}{|z|^{2k}},
 \]
 that is,
 \[ \varphi=\frac{1}{2}(\ln \frac{(K-\mu)^2}{|z|^{2k}}+\ln \frac{4(-2\mu-K)|\Phi|^2}{3}).
 \]
 $\dps{\lim_{z\rightarrow 0}\frac{\varphi+\ln |z|}{\ln |z|}=0}$ implies
 \BE\label{cusplim'}
  \lim_{z\rightarrow 0}\frac{\ln (\mu-K)}{\ln |z|}=k-1.
 \EE
 Suppose that $k>1$. Multiply both sides of (\ref{IKequ'}) by $\mu-K$ to get
 \bey\label{modifiedIKequ'}
 \nn (\mu-K)(-3\beta_1)[\ln (-2\mu-K)-\ln (\mu-K)+\frac{3\mu}{\mu-K}]= \\
  (\mu-K)(\lambda_{-1}\ln |z|^2+2Re(f_1)+c).
 \eey
 Then let $z\rightarrow 0$ on both sides of (\ref{modifiedIKequ'}) and take limits to get
 \BE\label{lim1}
 \lim_{z \rightarrow 0} (\mu-K)2Re(f_1)=\lim_{z\rightarrow 0}(\mu-K)(\frac{f_2}{z^{k-1}}+
 \frac{\overline{f_2}}{\bar{z}^{k-1}})= -9\beta_1 \mu.
 \EE
 Suppose
 \[ f_2(z)=\nu_0+zf_3(z),
 \]
 where $\nu_0$ is a nonzero complex number and $f_3(z)$ is a holomorphic function on $D$. Then
 \[ \frac{f_2(z)}{z^{k-1}}=\frac{\nu_0}{z^{k-1}}+\frac{f_3(z)}{z^{k-2}}~~\text{on}~~D \setminus \{0\}.
 \]
 We claim that
 \BE\label{lim2}
 \lim_{z \rightarrow 0} \frac{\mu-K}{z^{k-2}}=0.
 \EE
 Since $\dps{\lim_{z\rightarrow 0}\frac{\ln(\mu-K)}{\ln |z|}}=k-1$, if $z$ satisfies $0<|z|<\delta_1$
 ($\delta_1$ is small enough), $\dps{\frac{\ln (\mu-K)}{\ln |z|}}>d>k-2$. Then
 $\ln(\mu-K)<d \ln|z|$, that is, $\mu-K<|z|^d$, so
 \[ 0<\frac{\mu-K}{|z|^{k-2}}<\frac{|z|^d}{|z|^{k-2}}.
 \]
 Hence (\ref{lim2}) holds. By (\ref{lim2}), we also get
 \BE\label{lim3}
   \lim_{z \rightarrow 0} \frac{\mu-K}{\bar{z}^{k-2}}=0.
 \EE
 Then by (\ref{lim1}), (\ref{lim2}) and (\ref{lim3}),
 \BE\label{lim4}
 \lim_{z \rightarrow 0}(\mu-K)(\frac{\nu_0}{z^{k-1}}+\frac{\overline{\nu_0}}{\bar{z}^{k-1}})=
 -9\beta_1 \mu.
 \EE
 Let $\nu_0=\nu_1+\sqrt{-1}\nu_2(\nu_1,~\nu_2 \in \mathds{R})$ and $z=re^{\sqrt{-1}\theta}
 (0<r<\delta_2,~0 \leq \theta<2\pi)$.
 Then
 \[ \frac{\nu_0}{z^{k-1}}+\frac{\overline{\nu_0}}{\bar{z}^{k-1}}=\frac{2}{r^{k-1}}
 [\nu_1 \cos (k-1)\theta+\nu_2 \sin (k-1)\theta]
 \]
 and (\ref{lim4}) is
 \BE\label{lim5}
 \lim_{r \rightarrow 0}(\mu-K)\frac{2}{r^{k-1}}[\nu_1 \cos (k-1)\theta+\nu_2 \sin (k-1)\theta]
 =-9\beta_1 \mu.
 \EE
 Now fix a $\theta_0 \in [0,2\pi)$ such that
 \[\left\{
   \begin{array}{ll}
     \cos (k-1)\theta_0=\cfrac{-\nu_2}{\sqrt{\nu_1^2+\nu_2^2}}~, \\
     \sin (k-1)\theta_0=\cfrac{\nu_1}{\sqrt{\nu_1^2+\nu_2^2}}~.
   \end{array}
 \right.
 \]
 Then
 \[\lim_{r \rightarrow 0}[\mu-K(r\cos \theta_0,r \sin \theta_0)]\frac{2}{r^{k-1}}[\nu_1 \cos (k-1)
\theta_0+\nu_2 \sin (k-1)
\theta_0] =-9\beta_1 \mu \neq 0,
 \]
 however, since $\nu_1 \cos (k-1)\theta_0+\nu_2 \sin (k-1)\theta_0=0$,
 \[ \lim_{r \rightarrow 0}[\mu-K(r\cos \theta_0,r \sin \theta_0)]\frac{2}{r^{k-1}}[\nu_1 \cos (k-1)
\theta_0+\nu_2 \sin (k-1)
\theta_0] =0.
 \]
 It is a contradiction. \par
Consequently $k=1$, that is, $p_i$ is a simple pole of $\omega$.
 Further divide both sides of (\ref{IKequ'}) by $\ln |z|$, let $z \rightarrow 0$ and take limits to get
 \[ \lim_{z\rightarrow 0} \frac{\frac{9}{2}\beta_1 \mu}{(K-\mu)\ln |z|}=\lambda_{-1}.
 \]
 Therefore $Res_{p_i}(\omega)=\lambda_{-1}<0$. We prove the lemma.
 \end{proof}

 \begin{lemma}\label{boundlemma2}
 If $\forall p \in \Sigma',~K(p)<\mu$, then $\mathfrak{S}=\varnothing$.
 \end{lemma}
 \begin{proof}
 If $\mathfrak{S} \neq \varnothing$, pick a $e_s$ and let $(U,z)$ be a \lccc around $e_s$ such that
  $U \setminus \{e_s\} \subset \Sigma'$, $z(U)$ is a disk $D$ and $z(e_s)=0$.
  Suppose $g=e^{2\varphi}|dz|^2$ on $U$.
  Then $F=4e^{-2\varphi}K_{\bar{z}}$ is a holomorphic function on $D$, $0$ is the unique zero
  of $F$ on $D$ and $\omega=\cfrac{dz}{F}$ on $U \setminus \{e_s\}$. Suppose $\cfrac{1}{F}$
  has the following expression on $D \setminus \{0\}$:
  \BE\label{laurentexpr5}
   \frac{1}{F}=\frac{\lambda_{-k}}{z^k}+\cdots+\frac{\lambda_{-2}}{z^2}+\frac{\lambda_{-1}}{z}+
   \sum_{m=0}^\infty \lambda_m z^m=\frac{\Phi(z)}{z^k},
  \EE
  where $\Phi(z)$ is a holomorphic function on $D$ with $\Phi(0)=\lambda_{-k}\neq 0$. Then
  \BE\label{omegaexpr5}
   \omega=\frac{\lambda_{-1}}{z}dz+df_1,
  \EE
  where $f_1=\cfrac{f_2}{z^{k-1}}$ and $f_2$ is a holomorphic function on $D$ with $f_2(0) \neq 0$. \par
 On one hand, on $D \setminus \{0\}$
 \[e^{2\varphi}=-\frac{4}{3}(K-\mu)^2(K+2\mu)\frac{|\Phi|^2}{|z|^{2k}},
 \]
 then
  \[ \lim_{z \rightarrow 0}(-2\mu-K)(K-\mu)^2=0.
 \]
 Since $\forall p \in \Sigma',~K(p)<\mu$, we get
 \[ \lim_{z \rightarrow 0}K=K(e_s)=\mu.
 \]
 Further we obtain
 \BE\label{lims}
  \lim_{z \rightarrow 0}\frac{\mu-K}{|z|^k}=A,~A>0.
 \EE

 On the other hand, we substitute $\omega=\cfrac{\lambda_{-1}}{z}dz+df_1$ and
 $-\cfrac{K^3}{3}+CK+C'=-\cfrac{1}{3}(K-\mu)^2(K+2\mu)$ into (\ref{dKequ}) and integrate to
 get on $D \setminus \{0\}$
 \bey\label{IKequs}
 \nn (-\frac{1}{3\mu^2})[\ln (-2\mu-K)-\ln (\mu-K)-\frac{3\mu}{K-\mu}]= \\
 \lambda_{-1}\ln |z|^2+2Re(f_1)+c,
 \eey
 where $\lambda_{-1},~c \in \mathds{R}$. Then multiply both sides of (\ref{IKequs}) by
 $|z|^k$, let $z \rightarrow 0$ and take limits. By (\ref{lims}), the limit of the left side is
 a nonzero real number. The limit of the right side is $0$. It is a contradiction.
 Therefore we prove the lemma.
 \end{proof}

 \begin{lemma}\label{boundlemma3}
 If $\forall p \in \Sigma',~K(p)<\mu$, then each $q_j,~j=1,2,\cdots,J$, is not a pole of $\omega$.
 \end{lemma}
 \begin{proof}
  Suppose some $q_j$ is a pole of $\omega$. Let $(U,z)$ be a \lccc around $q_j$ such that
  $U \setminus \{q_j\} \subset \Sigma'$, $z(U)$ is a disk $D$ and $z(q_j)=0$.
   Suppose $g=e^{2\varphi}|dz|^2$ on $U \setminus \{q_j\}$.
  Then $F=4e^{-2\varphi}K_{\bar{z}}$ is a holomorphic function on $D$, $0$ is the unique zero
  of $F$ on $D$ and $\omega=\cfrac{dz}{F}$ on $U \setminus \{q_j\}$. Suppose $\cfrac{1}{F}$
  has the following expression on $D \setminus \{0\}$:
  \BE\label{laurentexpr6}
   \frac{1}{F}=\frac{\lambda_{-k}}{z^k}+\cdots+\frac{\lambda_{-2}}{z^2}+\frac{\lambda_{-1}}{z}+
   \sum_{m=0}^\infty \lambda_m z^m=\frac{\Phi(z)}{z^k},
  \EE
  where $\Phi(z)$ is a holomorphic function on $D$ with $\Phi(0)=\lambda_{-k}\neq 0$. Then
  \BE\label{omegaexpr6}
   \omega=\frac{\lambda_{-1}}{z}dz+df_1,
  \EE
  where $f_1=\cfrac{f_2}{z^{k-1}}$ and $f_2$ is a holomorphic function on $D$ with $f_2(0) \neq 0$. \par
 Then similar to above, we can get on $D \setminus \{0\}$
 \bey\label{IKequc}
 \nn (-\frac{1}{3\mu^2})[\ln (-2\mu-K)-\ln (\mu-K)-\frac{3\mu}{K-\mu}]= \\
 \lambda_{-1}\ln |z|^2+2Re(f_1)+c,
 \eey
 where $\lambda_{-1},~c \in \mathds{R}$. \par
 On the other hand, since $q_j$ is a conical singularity of $g$ with the singular angle $2 \pi \alpha_j$,
 we suppose
 \[ g=e^{2\varphi}|dz|^2=\frac{h}{|z|^{2-2\alpha_j}}|dz|^2~~\text{on}~~D \setminus \{0\},
 \]
 where $h$ is a positive continuous function on $D$. By (\ref{gequ}),
 \[ \frac{h}{|z|^{2-2\alpha_j}}=-\frac{4}{3}(K-\mu)^2(K+2\mu)\frac{|\Phi|^2}{|z|^{2k}}.
 \]
 Then
  \[ \lim_{z \rightarrow 0}(-2\mu-K)(K-\mu)^2=0.
 \]
 Since $\forall p \in \Sigma'$, $K(p)<\mu$,
 \[ \lim_{z \rightarrow 0}K=\mu.
 \]
 Further we get
 \[ \lim_{z \rightarrow 0} \frac{\mu-K}{|z|^{k-1+\alpha_j}}=A,~A>0.
 \]
 Then multiply both sides of (\ref{IKequc}) by $|z|^{k-1+\alpha_j}$, let $z \rightarrow 0$ and take limits.
 The limit of the left side is a nonzero real number. The limit of the right side
 is $0$. It is a contradiction. Therefore we prove the lemma.
 \end{proof}
 Now we can get the following theorem:
 \begin{theorem}\label{boundth}
 $\forall p \in \Sigma'$, $\mu<K(p)<-2\mu$.
 \end{theorem}
 \begin{proof}
 Otherwise $\forall p \in \Sigma'$, $K(p)<\mu$. Then by Lemma \ref{boundlemma1}, each $p_i$ is
 a simple pole of $\omega$ and the residue of $\omega$ at each $p_i$ is a negative real number.
 Since $\omega$ is a meromorphic 1-form on $\Sigma$, the sum of the residues of $\omega$ is zero.
That means $\omega$ must have other poles besides $p_1,p_2,\cdots,p_I$. Obviously the set of these
poles of $\omega$ besides $p_1,p_2,\cdots,p_I$ is a subset of $\{e_1,e_2,\cdots,e_S,q_1,q_2,\cdots,q_J\}$.
By Lemma \ref{boundlemma2}, $\mathfrak{S}=\varnothing$ and by Lemma \ref{boundlemma3}, each $q_j$
is not a pole of $\omega$. It is a contradiction. Hence $\forall p \in \Sigma'$, $\mu<K(p)<-2\mu$.
We prove the theorem.
 \end{proof}
Next we will get a theorem about the cusp singularities:
\begin{theorem}\label{cuspth2}
 Each $p_i,~i=1,2,\cdots,I,$ is a simple pole of $\omega$ and the residue of $\omega$ at each $p_i$ is
a positive real number.
\end{theorem}
\begin{proof}
 Using the similar argument in the proof of Lemma \ref{boundlemma1}(the only difference is in the
 proof of Lemma \ref{boundlemma1} $K<\mu$, but here $K>\mu$), we can prove this theorem.
\end{proof}
Then a theorem about $\mathfrak{S}$ will be obtained.
\begin{theorem}\label{sth}
If $\mathfrak{S}\neq \varnothing$, then each $e_s,~s=1,2,\cdots,S$, is a simple pole of $\omega$,
the residue of $\omega$ at each $e_s$, $s=1,2,\cdots,S$, is $-\cfrac{1}{3\mu^2}$ and $K(e_s)=-2\mu$, $s=1,2,\cdots,S$.
\end{theorem}
\begin{proof}
 Suppose that $\mathfrak{S}\neq \varnothing$. Pick any $e_s$ and let $(U,z)$ be a \lccc around $e_s$
such that $U \setminus \{e_s\} \subset \Sigma'$, $z(U)$ is a disk $D$ and $z(e_s)=0$.
Suppose $g=e^{2\varphi}|dz|^2$ on $U$.
  Then $F=4e^{-2\varphi}K_{\bar{z}}$ is a holomorphic function on $D$, $0$ is the unique zero
  of $F$ on $D$ and $\omega=\cfrac{dz}{F}$ on $U \setminus \{e_s\}$. Suppose $\cfrac{1}{F}$
  has the following expression on $D \setminus \{0\}$:
  \BE\label{laurentexpr7}
   \frac{1}{F}=\frac{\lambda_{-k}}{z^k}+\cdots+\frac{\lambda_{-2}}{z^2}+\frac{\lambda_{-1}}{z}+
   \sum_{m=0}^\infty \lambda_m z^m=\frac{\Phi(z)}{z^k},
  \EE
  where $\Phi(z)$ is a holomorphic function on $D$ with $\Phi(0)=\lambda_{-k}\neq 0$. Then
  \BE\label{omegaexpr7}
   \omega=\frac{\lambda_{-1}}{z}dz+df_1,
  \EE
  where $f_1=\cfrac{f_2}{z^{k-1}}$ and $f_2$ is a holomorphic function on $D$ with $f_2(0) \neq 0$. \par
  Then first on $D \setminus \{0\}$
 \bey\label{IKequss}
 \nn (-\frac{1}{3\mu^2})[\ln(-2\mu-K)-\ln(K-\mu)-\frac{3\mu}{K-\mu}]= \\
 \lambda_{-1}\ln |z|^2+2Re(f_1)+c,
 \eey
 where $\lambda_{-1},~c \in \mathds{R}$. \par
On the other hand, on $D \setminus \{0\}$
 \[ e^{2\varphi}=-\frac{4}{3}(K-\mu)^2(K+2\mu)\frac{|\Phi|^2}{|z|^{2k}}.
 \]
 Then
\[ \lim_{z \rightarrow 0}(K-\mu)^2(-2\mu-K)=0.
\]
Since $K$ is continuous at $e_s$, $K(e_s)=\mu$ or $K(e_s)=-2\mu$. If
$K(e_s)=\mu$, then
\[ \lim_{z \rightarrow 0} \frac{K-\mu}{|z|^k}=A_1,~A_1>0.
\]
Multiply both sides of (\ref{IKequss}) by $|z|^k$, let $z \rightarrow 0$ and take limits.
The limit of the left side is a nonzero real number. The limit of
the right side is $0$. It is a contradiction. Therefore $K(e_s)=-2\mu$ and
\[ \lim_{z \rightarrow 0} \frac{-2\mu-K}{|z|^{2k}}=A_2,~A_2>0.
\]
Then
\BE\label{limss}
 \lim_{z \rightarrow 0} \frac{\ln (-2\mu-K)}{\ln |z|}=2k.
\EE
 \par
If $k>1$, multiply both sides of (\ref{IKequss}) by $z^{k-1}$ to get
\bey\label{IKequss'}
\nonumber z^{k-1}(-\frac{1}{3\mu^2})[\ln(-2\mu-K)-\ln(K-\mu)-\frac{3\mu}{K-\mu}]=\\
z^{k-1}\lambda_{-1}\ln|z|^2+f_2+\overline{f_2}\frac{z^{k-1}}{\bar{z}^{k-1}}+cz^{k-1}.
\eey
Then let $z \rightarrow 0$ on both sides of (\ref{IKequss'}) and take limits. By (\ref{limss}), we get
\[ \lim_{z \rightarrow 0}\frac{z^{k-1}}{\bar{z}^{k-1}}=A_3,~A_3\neq 0.
\]
It is impossible. Hence $k=1$, that is, $e_s$ is a simple pole of $\omega$. Then divide both
sides of (\ref{IKequss}) by $\ln |z|$, let $z \rightarrow 0$ and take limits to get
\[ -\frac{1}{3\mu^2}\lim_{z \rightarrow 0}\frac{\ln(-2\mu-K)}{\ln |z|}=2\lambda_{-1}.
\]
By (\ref{limss}), $\lambda_{-1}=Res_{e_s}(\omega)=-\cfrac{1}{3\mu^2}$. Therefore we
prove the theorem.
\end{proof}

Next we will consider the conical singularities of $g$, $q_1,q_2,\cdots,q_J$. First we get the following
result: if the singular angle of $g$ at $q_j$ is $2\pi$, then $q_j$ is a regular point of $g$.
Before the result, we will give a lemma:
\begin{lemma}\label{analysislemma}
Suppose that $\Omega$ is a domain in $\mathbb{R}^N$ and $0 \in \Omega$. Let $f$ be a continuous function
on $\Omega$ and $f|_{\Omega \setminus \{0\}} \in C^1(\Omega\setminus \{0\})$. If there exist
$g_1,g_2,\cdots,g_N
\in C^0(\Omega)$ such that $\cfrac{\partial f }{\partial x_\nu}=g_\nu$ holds on $\Omega \setminus
\{0\},~\forall \nu,~\nu=1,2,\cdots,N$, then $f \in C^1(\Omega)$.
\end{lemma}
\begin{proof}
 In fact we only need to prove that $f$ has partial derivatives at $0$ and $\cfrac{\partial f}
{\partial x_\nu}(0)=g_\nu(0), \forall \nu,~\nu=1,2,\cdots,N.$ \par
Pick any $\nu$. Suppose that $(\underbrace{0,\cdots,
t , \cdots,0}_\nu) \in \Omega$ as $t \in [-\Delta,\Delta](\Delta>0)$. Define $\Gamma(t)=
f(\underbrace{0,\cdots, t , \cdots,0}_\nu),~t\in [0,\Delta]$. Then $\Gamma(t)$ is continuous on
$[0,\Delta]$ and $\Gamma'(t)=g_\nu(\underbrace{0,\cdots, t , \cdots,0}_\nu), \forall t \in (0,\Delta)$.
Then by Newton-Leibnitz formula
\[ \int_0^t g_\nu(\underbrace{0,\cdots, \tau , \cdots,0}_\nu)d\tau=\Gamma(t)-\Gamma(0),~0<t<\Delta.
\]
Hence
\[ \lim_{t \rightarrow 0^+}\frac{\Gamma(t)-\Gamma(0)}{t}=\lim_{t \rightarrow
0^+} \frac{f(\underbrace{0,\cdots, t , \cdots,0}_\nu)-f(0)}{t}=g_\nu(0).
\]
Similarly,
\[ \lim_{t \rightarrow 0^-} \frac{f(\underbrace{0,\cdots, t , \cdots,0}_\nu)-f(0)}{t}=
g_\nu(0).
\]
Therefore
\[ \frac{\partial f}{\partial x_\nu}(0)=g_\nu(0).
\]
Then we prove the lemma.
\end{proof}
\begin{proposition}\label{regprop}
If the singular angle of $g$ at $q_j$ is $2\pi$, then $q_j$ is a regular point of $g$.
\end{proposition}
\begin{proof}
Let $(U,z)$ be a local complex coordinate chart around $q_j$ such that $U \setminus
\{q_j\} \subset \Sigma'$, $z(U)$ is a disk $D$ and $z(q_j)=0$. Since $q_j$ is a conical
singularity of $g$ with the angle $2\pi$, we suppose $g=h |dz|^2$ on $U \setminus
\{q_j\}$, where $h$ is a positive continuous function on $U$ and is smooth on $U
\setminus \{q_j\}$. Then $F=\cfrac{4K_{\bar{z}}}{h}$ is a nonvanishing holomorphic
function on $U \setminus \{q_j\}$ and $\omega=\cfrac{dz}{F}$ on $U \setminus \{q_j\}$.
By (\ref{gequ}), on $D \setminus \{0\}$
\[h=-\cfrac{4}{3}(K-\mu)^2(K+2\mu)\frac{1}{|F|^2}.
\]
 By Theorem \ref{boundth}, $-(K-\mu)^2(K+2\mu)$ is bounded
on $D \setminus \{0\}$,
so $q_j$ is not a zero of $\omega$. \par
If $q_j$ is a pole of $\omega$, suppose $\cfrac{1}{F}$ has the following expression on
$D \setminus \{0\}$,
\[ \frac{1}{F}=\frac{\lambda_{-k}}{z^k}+\cdots+\frac{\lambda_{-2}}{z^2}+\frac{\lambda_{-1}}{z}+
   \sum_{m=0}^\infty \lambda_m z^m=\frac{\Phi(z)}{z^k},
\]
where $\Phi(z)$ is a holomorphic function on $D$ with $\Phi(0)=\lambda_{-k}\neq 0$. Therefore
\[
\omega=\frac{\lambda_{-1}}{z}dz+df_1,
\]
where $f_1=\cfrac{f_2}{z^{k-1}}$ and $f_2$ is a holomorphic function on $D$ with $f_2(0) \neq 0$.
Then we get on $D \setminus \{0\}$
\bey\label{IKequc1}
 \nn (-\frac{1}{3\mu^2})[\ln(-2\mu-K)-\ln(K-\mu)-\frac{3\mu}{K-\mu}]=\\
\lambda_{-1}\ln |z|^2+2Re(f_1)+c_1,
\eey
where $\lambda_{-1},~c_1\in \mathds{R}$. On the other hand, on $D \setminus \{0\}$
\[ h=-\frac{4}{3}(K-\mu)^2(K+2\mu)\frac{|\Phi|^2}{|z|^{2k}},
\]
so
\[ \lim_{z \rightarrow 0}(-2\mu-K)(K-\mu)^2=0.
\]
If $\dps{\limsup_{z \rightarrow 0}}K=-2\mu$ and $\dps{\liminf_{z \rightarrow 0}K=\mu}$, then
there exist two sequences $\{x_{\ell}\}$, $\{x_{\ell}'\} \subset D \setminus \{0\}$ such that
$x_{\ell} \rightarrow 0$, $x_{\ell}' \rightarrow 0$ as $\ell \rightarrow \infty$ and $\dps{\lim_{\ell \rightarrow
\infty}}K(x_{\ell})=-2\mu$,
$\dps{\lim_{\ell \rightarrow \infty}K(x_{\ell}')}=\mu$. Pick any $\hat{\mu}$ such that $\mu < \hat{\mu}< -2\mu$.
As $\ell$ is big enough, $K(x_{\ell})>\hat{\mu}$ and $K(x_{\ell}')<\hat{\mu}$. Then there exists $y_{\ell}$
which satisfies $\min \{|x_{\ell}|, |x_{\ell}'|\} \leq |y_{\ell}| \leq \max \{|x_\ell|, |x_{\ell}'|\}$ such that
$K(y_\ell)=\hat{\mu}$.
Obviously $y_\ell \rightarrow 0$ as $\ell \rightarrow \infty$, so
\[ \lim_{\ell \rightarrow \infty} (-2\mu-K(y_\ell))(K(y_\ell)-\mu)^2=0,
\]
which means $(-2\mu-\hat{\mu})(\hat{\mu}-\mu)^2=0$. It is impossible. Therefore
$\dps{\lim_{z \rightarrow 0}K=-2\mu}$ or $\dps{\lim_{z \rightarrow 0}K=\mu}$. If $\dps
{\lim_{z \rightarrow 0}}K=\mu$, we can use the same method in the proof of Theorem \ref{sth} to
get a contradiction. Then $\dps{\lim_{z \rightarrow 0}K=-2\mu}$, which means $K$ is continuous
on $D$. Further
\[ \lim_{z \rightarrow 0} \frac{-2\mu-K}{|z|^{2k}}=A,~A>0,
\]
which implies
\[ \lim_{z \rightarrow 0}\frac{\ln(-2\mu-K)}{\ln |z|}=2k.
\]
If $k>1$, we can also use the same method in the proof of Theorem \ref{sth} to
get a contradiction. Hence $k=1$. Then divide both sides of (\ref{IKequc1}) by $\ln |z|$, let
$z \rightarrow 0$ and take limits to get
$ \lambda_{-1}=\cfrac{-1}{3\mu^2}.
$
Next by (\ref{IKequc1}), on $D \setminus \{0\}$
\[ -2\mu-K=|z|^2 (K-\mu)e^{\frac{3\mu}{K-\mu}}f_3,
\]
where $f_3$ is a positive smooth function on $D$. Therefore on $D \setminus \{0\}$
\beyy
h &=&\frac{4}{3}(K-\mu)^2|z|^2(K-\mu)e^{\frac{3\mu}{K-\mu}}f_3 \frac{|\Phi|^2}{|z|^2}\\
  &=&(K-\mu)^3 e^{\frac{3\mu}{K-\mu}}f_4,
\eeyy
where $f_4$ is a positive smooth function on $D$. Then consider the following system of equations on
$D \setminus \{0\}$:
\BE\label{sys1}
\begin{cases}
 K_{\bar{z}}=\frac{1}{4}F h, \\
h=(K-\mu)^3 e^{\frac{3\mu}{K-\mu}}f_4.
\end{cases}
\EE
By Lemma \ref{analysislemma} and the first equation of (\ref{sys1}),
$K \in C^1(D)$. Then using the bootstrap technique to (\ref{sys1}),
we get $K \in C^{\infty}(D)$. Finally by the second equation of (\ref{sys1}), $h
\in C^{\infty}(D)$, which means $q_j$ is a regular point of $g$.  \par
If $q_j$ is a regular point of $\omega$(neither a pole nor a zero of $\omega$), then $\omega=
f_5'(z)dz=d f_5$, where $f_5$ is a holomorphic function on $D$. Therefore on $D \setminus \{0\}$
\BE\label{zeroequ}
 -\frac{1}{3\mu^2}[\ln(-2\mu-K)-\ln(K-\mu)-\frac{3\mu}{K-\mu}]=f_5+\overline{f_5}+c_2,
\EE
where $c_2 \in \mathds{R}$. Let
\[\sigma(t)=\ln(-2\mu-t)-\ln(t-\mu)-\frac{3\mu}{t-\mu}, ~t \in (\mu,-2\mu).
\]
Then $\sigma(t)$ has the following property: $\forall x \in \mathds{R}, \exists!~t \in (\mu,
-2\mu),~s.t.~\sigma(t)=x$. The reasons for the existence are the continuity of $\sigma(t)$ and
\[ \lim_{t \rightarrow (-2\mu)^-} \sigma(t)=-\infty~~\text{,}~~
\lim_{t \rightarrow \mu^+} \sigma(t)=+\infty.
\]
The reason for the uniqueness is since
\[ \sigma'(t)=\frac{1}{t+2\mu}-\frac{1}{t-\mu}+\frac{3\mu}{(t-\mu)^2}=\frac{9\mu^2}{(t-\mu)^2
(t+2\mu)},
\]
$\sigma'(t)\neq 0$, $\forall t \in (\mu,-2\mu)$. By the property of $\sigma(t)$, we can define
a function $T$ on $D$ such that $\mu < T< -2\mu$ and $\sigma(T)=(-3\mu^2)(f_5+\overline{f_5}
+c_2)$. Then by the implicit theorem, $T \in C^\infty(D)$. By (\ref{zeroequ}), $K=T$ on
$D \setminus \{0\}$, so $K$ can be smoothly extended to $q_j$. By (\ref{gequ}), on $D \setminus \{0\}$
\[
h=-\frac{4}{3}(K-\mu)^2(K+2\mu)|f_5'(z)|^2.
\]
Since $K$ is smooth on $D$ and $q_j$ is a regular point of $\omega$, $h$ is smooth on
$D$, which also means $q_j$ is a regular point of $g$. Then we prove the proposition.
\end{proof}
\begin{remark}
 By Proposition \ref{regprop}, we suppose that the singular angle of $g$ at each $q_j$, $j=1,2,\cdots,J$,
 is not $2\pi$. That is why we suppose $\alpha_j \neq 1$, $j=1,2,\cdots,J$, in Theorem \ref{mainth}.
\end{remark}
Next we will get the following theorem for the conical singularities $q_1,q_2,\cdots,q_J$:
\begin{theorem}\label{conicalth}
Each $q_j$, $j=1,2,\cdots,J$, is a pole or a zero of $\omega$. If $q_j$ is a zero of $\omega$,
then $\alpha_j$ must be an integer, the order of $\omega$ at $q_j$ is $\alpha_j-1$, $K$ can
be smoothly extended to $q_j$ with $\mu<K(q_j)<-2\mu$ and $dK(q_j)=0$. If $q_j$ is a pole of $\omega$,
then $q_j$ is a simple pole of $\omega$, the residue of $\omega$ at $q_j$ is $-\cfrac{\alpha_j}{3\mu^2}$
and $K$ can be continuously extended to $q_j$ with $K(q_j)=-2\mu$.
\end{theorem}
\begin{proof}
Suppose that $q_j$ is a regular point of $\omega$. Let $(W,\xi)$ be a \lccc around $q_j$
such that $W\setminus \{q_j\} \subset \Sigma'$, $\xi(W)$ is a disk $\widetilde{D}$ and $\xi(q_j)=0$
. Assume $\omega=\rho(\xi)d\xi$ on $W$, where
$\rho(\xi)$ is a holomorphic function on $\widetilde{D}$ with $\rho(0)\neq 0$. Then there exists a
positive continuous function $\tilde{h}$ on $\widetilde{D}$ such that on $\widetilde{D} \setminus \{0\}$
\BE\label{noreg}
\frac{\tilde{h}}{|\xi|^{2-2\alpha_j}}=-\frac{4}{3}(K-\mu)^2(K+2\mu)|\rho(\xi)|^2.
\EE
Then we can also use the argument in the proof of Proposition \ref{regprop} to get $K$ can be
smoothly extended to $q_j$ with $\mu<K(q_j)<-2\mu$. If $\alpha_j<1$, let $\xi\rightarrow 0$ and
take limits on both sides of (\ref{noreg}). The limit of the left side is $+\infty$ and
the limit of the right side is a nonzero real number. It is a contradiction. If $\alpha_j>1$,
let $\xi\rightarrow 0$ and take limits on both sides of (\ref{noreg}). The limit of the left side is
$0$ and the limit of the right side is a nonzero real number. It is also a contradiction.
Hence $q_j$ is not a regular point of $\omega$. \par
Suppose $q_j$ is a zero of $\omega$. Let $(Y,\zeta)$ be a \lccc around $q_j$ such that $Y \setminus \{q_j\}
\subset \Sigma'$, $\zeta(Y)$ is a disk $\widehat{D}$ and $\zeta(q_j)=0$.
Assume on $Y$
\[\omega=\zeta^\gamma H_1 d\zeta=dH_2,
\]
where $\gamma$ is the order of $\omega$ at $q_j$, $H_1$ is a holomorphic function on $\widehat{D}$ with
$H_1(0)\neq 0$ and $H_2$ is a holomorphic function on $\widehat{D}$. Then also by the argument in the proof
of Proposition \ref{regprop}, we have $K$ can be smoothly extended to $q_j$ with $\mu<K(q_j)<-2\mu$.
By (\ref{dKequ}), $dK(q_j)=0$. By (\ref{gequ}), there exists a positive continuous function
$\hat{h}$ on $\widehat{D}$ such that on $\widehat{D} \setminus \{0\}$
\[\frac{\hat{h}}{|\zeta|^{2-2\alpha_j}}=-\frac{4}{3}(K-\mu)^2(K+2\mu)|\zeta|^{2\gamma}
|H_1|^2.
\]
Therefore $\gamma=\alpha_j-1$. \par
Suppose $q_j$ is a pole of $\omega$. Let $(U,z)$ be a \lccc around $q_j$ such that $U \setminus
\{q_j\} \subset \Sigma'$, $z(U)$ is a disk $D$ and $z(q_j)=0$. Suppose
  $g=e^{2\varphi}|dz|^2$ on $U \setminus \{q_j\}$.
  Then $F=4e^{-2\varphi}K_{\bar{z}}$ is a holomorphic function on $D$, $0$ is the
unique zero of $F$ on $D$ and $\omega=\cfrac{dz}{F}$ on $U \setminus \{q_j\}$. Suppose $\cfrac{1}{F}$
  has the following expression on $D \setminus \{0\}$:
  \BE\label{laurentexpr8}
   \frac{1}{F}=\frac{\lambda_{-k}}{z^k}+\cdots+\frac{\lambda_{-2}}{z^2}+\frac{\lambda_{-1}}{z}+
   \sum_{m=0}^\infty \lambda_m z^m=\frac{\Phi(z)}{z^k},
  \EE
  where $\Phi(z)$ is a holomorphic function on $D$ with $\Phi(0)=\lambda_{-k}\neq 0$. Then
  \BE\label{omegaexpr8}
   \omega=\frac{\lambda_{-1}}{z}dz+df_1,
  \EE
  where $f_1=\cfrac{f_2}{z^{k-1}}$ and $f_2$ is a holomorphic function on $D$ with $f_2(0) \neq 0$.
Then on $D \setminus \{0\}$
\bey\label{conicalequ2}
\nn (-\frac{1}{3\mu^2})[\ln(-2\mu-K)-\ln(K-\mu)-\frac{3\mu}{K-\mu}]=\\
 \lambda_{-1}\ln |z|^2+2Re(f_1)+c .
\eey
where $\lambda_{-1}, c \in \mathds{R}$. \par
On the other hand, there exists a positive continuous function $h$ on $D$ such that on $D \setminus
\{0\}$
\[ \frac{h}{|z|^{2-2\alpha_j}}=-\frac{4}{3}(K-\mu)^2(K+2\mu)\frac{|\Phi|^2}{|z|^{2k}}.
\]
Then
\[ \lim_{z \rightarrow 0}(K-\mu)^2(K+2\mu)=0.
\]
By the same argument in the proof of Proposition \ref{regprop},
\[ \lim_{z \rightarrow 0} K=\mu~~\text{or}~~\lim_{z \rightarrow 0} K=-2\mu.
\]
If $\dps{\lim_{z \rightarrow 0} K=\mu}$, then
\[ \lim_{z \rightarrow 0} \frac{K-\mu}{|z|^{k+\alpha_j-1}}=A_1,~A_1>0.
\]
Multiply both sides of (\ref{conicalequ2}) by $|z|^{k+\alpha_j-1}$, let $z \rightarrow 0$ and take limits.
The limit of the left side is a nonzero real number and the limit of the right side is $0$. It is a contradiction.
Hence we have $\dps{\lim_{z \rightarrow 0}K=-2\mu}$. Then
\[ \lim_{z \rightarrow 0} \frac{-2\mu-K}{|z|^{2k+2\alpha_j-2}}=A_2,~A_2>0,
\]
which implies
\[\lim_{z \rightarrow 0} \frac{\ln (-2\mu-K)}{\ln |z|}=2k+2\alpha_j-2.
\]
If $k>1$, we can also use the argument in the proof of Theorem \ref{sth} to get a contradiction.
Therefore $k=1$, that is, $q_j$ is a simple pole of $\omega$. Then divide both sides of (\ref{conicalequ2})
by $\ln |z|$, let $z \rightarrow 0$ and take limits. The limit of the left side is
$-\cfrac{2\alpha_j}{3\mu^2}$ and the limit of the right side is $2\lambda_{-1}$, so
\[Res_{q_j}(\omega)=\lambda_{-1}=-\frac{\alpha_j}{3\mu^2}.
\]
Then we prove the theorem.
\end{proof}
Therefore we finish the proof of the necessity of Theorem \ref{mainth}.
\begin{remark} By (3) in Theorem \ref{ehmcuspth} and Theorem \ref{conicalth}, $K$ is a continuous function on
$\Sigma$. By the assumption in Theorem \ref{mainth} and Theorem \ref{conicalth}, $q_1,q_2,\cdots,q_L$ which are
the saddle points of $K$ are the zeros of $\omega$ and $q_{L+1},q_{L+2},\cdots,q_J$ are the poles of $\omega$.
\end{remark}
In the following, we will give formulas for
\[ \mathcal{C}_n=\int_{\Sigma'} K^n dg,~n=0,1,2,\cdots
\]
Obviously, $\mathcal{C}_0$ is the area of $g$, $\mathcal{C}_1$ is related to the generalized Gauss-
Bonnet formula and $\mathcal{C}_2$ is the Calabi energy of $g$. First
\beyy
\mathcal{C}_n &=& \int_{\Sigma'} K^n dg\\
 &=& \frac{\sqrt{-1}}{2}\int_{\Sigma'} K^n \frac{-4}{3}(K-\mu)^2(K+2\mu)\omega \wedge \bar{\omega}\\
 &=& 2\sqrt{-1}\int_{\Sigma'} K^n \partial K\wedge \bar{\omega}\\
 &=& \frac{2\sqrt{-1}}{n+1} \int_{\Sigma'} d(K^{n+1} \bar{\omega}) \\
 &=& \lim_{\varepsilon \rightarrow 0} \frac{2\sqrt{-1}}{n+1}\int_{\Sigma \setminus (\cup_{i=1}^I
 D_\varepsilon (p_i)\bigcup \cup_{s=1}^S D_\varepsilon (e_s) \bigcup \cup_{j=1}^J D_\varepsilon
 (q_j))} d(K^{n+1}\bar{\omega}),
\eeyy
where $D_\varepsilon(p_i)(D_\varepsilon(e_s), D_\varepsilon(q_j))$ is a coordinate disk around $p_i(e_s, q_j)$
with the center $p_i(e_s, q_j)$ and the radius $\varepsilon$. By the Stokes formula,
\beyy
 \int_{\Sigma \setminus (\cup_{i=1}^I
 D_\varepsilon (p_i)\bigcup \cup_{s=1}^S D_\varepsilon (e_s) \bigcup \cup_{j=1}^J D_\varepsilon
 (q_j))} d(K^{n+1}\bar{\omega})=\\
-\sum_{i=1}^I \oint_{\partial D_\varepsilon (p_i)} K^{n+1} \bar{\omega}-
\sum_{s=1}^S \oint_{\partial D_\varepsilon (e_s)} K^{n+1} \bar{\omega}-\sum_{j=1}^J \oint_
{\partial D_\varepsilon (q_j)} K^{n+1} \bar{\omega},
\eeyy
where the directions of the integrations are anticlockwise. Consider
\[ \lim_{\varepsilon \rightarrow 0} \oint_{\partial D_\varepsilon (q_j)} K^{n+1} \bar{\omega}.
\]
If $q_j$ is a zero of $\omega$, suppose on $D_{\varepsilon_0}(q_j) $
\[ \omega=\rho_1(z)dz,
\]
where $\rho_1(z)$ is a holomorphic function on the coordinate disk $D_{\varepsilon_0}(q_j)$. Then
$\forall \varepsilon,~0<\varepsilon<\varepsilon_0$,
\[
\oint_{\partial D_\varepsilon (q_j)} K^{n+1} \bar{\omega}=
(-\sqrt{-1})\varepsilon \int_0^{2\pi} K^{n+1} \overline{\rho_1}e^{-\sqrt{-1}\theta}d\theta,
\]
where $z=re^{\sqrt{-1}\theta}$ on $D_{\varepsilon_0}(q_j)$. Since $K$ and $\overline{\rho_1}$ are
bounded around $q_j$,
\[ \lim_{\varepsilon \rightarrow 0}\oint_{\partial D_\varepsilon(q_j)} K^{n+1} \bar{\omega}=0.
\]
If $q_j$ is a pole of $\omega$, suppose on $D_{\varepsilon_1}(q_j) \setminus \{0\}$
\[ \omega=\frac{\lambda_{-1}}{z}dz+\rho_2(z)dz,
\]
where $\lambda_{-1}=Res_{q_j}(\omega)$ and $\rho_2(z)$ is a holomorphic function on the
coordinate disk $D_{\varepsilon_1}(q_j)$. Then $\forall \varepsilon,~0<\varepsilon<\varepsilon_1$,
\beyy \oint_{\partial D_\varepsilon (q_j)} K^{n+1} \bar{\omega} &=&
\oint_{\partial D_\varepsilon (q_j)} K^{n+1} (\frac{\lambda_{-1}}{\bar{z}}d\bar{z}+\overline{\rho_2(z)}
d\bar{z})\\
&=& (-\sqrt{-1}\lambda_{-1})\int_0^{2\pi} K^{n+1}d\theta+\oint_{\partial D_\varepsilon (q_j)}K^{n+1}
\overline{\rho_2}d\bar{z}.
\eeyy
Since
\[ \lim_{\varepsilon \rightarrow 0} \int_0^{2\pi} K^{n+1}d\theta=2\pi K(q_j)^{n+1} ~~\text{and}~~\lim_{\varepsilon
\rightarrow 0} \oint_{\partial D_\varepsilon (q_j)} K^{n+1} \overline{\rho_2} d\bar{z}=0,
\]
\beyy
 \lim_{\varepsilon \rightarrow 0} \oint_{\partial D_\varepsilon (q_j)} K^{n+1} \bar{\omega} &=&
(-\sqrt{-1}\lambda_{-1})2\pi K(q_j)^{n+1} \\
&=& (-2\pi \sqrt{-1})Res_{q_j}(\omega)(-2\mu)^{n+1}.
\eeyy
Similarly,
\[\lim_{\varepsilon \rightarrow 0} \oint_{\partial D_\varepsilon(e_s)} K^{n+1} \bar{\omega}=
(-2\pi \sqrt{-1})Res_{e_s}(\omega)(-2\mu)^{n+1},
\]
\[ \lim_{\varepsilon \rightarrow 0} \oint_{\partial D_\varepsilon(p_i)} K^{n+1} \bar{\omega}=
(-2\pi \sqrt{-1})Res_{p_i}(\omega)\mu^{n+1}.
\]
Therefore we obtain
\beyy
 \mathcal{C}_n = \frac{2\sqrt{-1}}{n+1} [(2\pi \sqrt{-1})\sum_{i=1}^I Res_{p_i}(\omega)\mu^{n+1}+\\ (2\pi
\sqrt{-1})\sum_{s=1}^S Res_{e_s}(\omega)(-2\mu)^{n+1}+
(2\pi \sqrt{-1})\sum_{j=L+1}^J Res_{q_j}(\omega)(-2\mu)^{n+1}] \\
=\frac{(-4\pi)}{n+1} \mu^{n+1} [\sum_{i=1}^I Res_{p_i} (\omega)+\sum_{s=1}^S Res_{e_s}(\omega)
(-2)^{n+1}+\sum_{j=L+1}^J Res_{q_j}(\omega)(-2)^{n+1}].
\eeyy
Since
\[ \sum_{i=1}^I Res_{p_i}(\omega)+\sum_{s=1}^S Res_{e_s}(\omega)+\sum_{j=L+1}^J Res_{q_j}(\omega)=0
\]
and
\beyy
 Res_{e_s}(\omega)=-\frac{1}{3\mu^2},~s=1,2,\cdots,S,\\
 Res_{q_j}(\omega)=-\frac{\alpha_j}{3\mu^2},~j=L+1,L+2,\cdots,J,
\eeyy
\BE\label{C_n}
 \mathcal{C}_n=\frac{2}{3(n+1)}\mu^{n-1}[(-2)^{n+1}-1]\alpha_{max},
\EE
where $\alpha_{max}=2\pi (S+\dps{\sum_{j=L+1}^J}\alpha_j)$ means the sum of the angles of the maximum points of
$K$. By (\ref{C_n}), $\mathcal{C}_n>0$, $n=0,1,2,\cdots$, and
\beyy
 \mathcal{C}_0=Area(g)=-\frac{2}{\mu}\alpha_{max},\\
 \mathcal{C}_1=\int_{\Sigma '}K dg=\alpha_{max}=2\pi(S+\dps{\sum_{j=L+1}^J}\alpha_j), \\
 \mathcal{C}_2=\mathcal{C}(g)=-2\mu \alpha_{max}.
\eeyy
\subsection{Proof of the sufficiency of the main theorem}
In this section, we will follow the steps in the proof of the sufficiency of the main theorem in \cite{CW2}.
Since $\omega+\bar{\omega}$ is exact on $\Sigma'=\Sigma \setminus \{ p_1,p_2,\cdots,p_I,q_1,q_2,\cdots,q_J,
e_1,e_2,\cdots,e_S\}$, we suppose that $\omega+\bar{\omega}=df_0$,
where $f_0$ is a smooth function on $\Sigma'$. And we let $\mu=-\cfrac{1}{\sqrt{-3\Lambda}}$. \par
\noindent {\bf Step 1.} Consider the equation on $\Sigma'$
\BE\label{dKequS}
 \begin{cases}
  \cfrac{(-3)dK}{(K-\mu)^2(K+2\mu)}=\omega+\bar{\omega} \\
 K(p_0)=K_0,~p_0 \in \Sigma',~\mu < K_0<-2\mu.
 \end{cases}
\EE

\noindent \textit{Claim 1:} (\ref{dKequS}) has a unique real smooth solution $K$ on $\Sigma'$ such that $\mu <K<-2\mu$.
\begin{proof}
 First
\[ \frac{1}{(K-\mu)^2(K+2\mu)}=\frac{1}{9\mu^2}[\frac{1}{K+2\mu}-\frac{1}{K-\mu}+\frac{3\mu}{(K-\mu)^2}].
\]
Since $\omega+\bar{\omega}=df_0$,
(\ref{dKequS}) is equivalent to
\BE\label{dKequS'}
 \begin{cases}
 [\cfrac{1}{K+2\mu}-\cfrac{1}{K-\mu}+\cfrac{3\mu}{(K-\mu)^2}]dK=d\cfrac{f_0}{\Lambda} \\
 K(p_0)=K_0.
 \end{cases}
\EE
Also let
\[ \sigma(t)=\ln(-2\mu-t)-\ln(t-\mu)-\frac{3\mu}{t-\mu},~t \in (\mu,-2\mu).
\]
Then by the same argument in the proof of Proposition \ref{regprop}, we can define a real function
$K^*$ on $\Sigma'$ such that $\mu<K^*<-2\mu$ and
\BE\label{IKequS}
 \sigma(K^*)=\frac{f_0}{\Lambda}+A_0,
\EE
where $A_0=\sigma(K_0)-\cfrac{f_0(p_0)}{\Lambda}$. By the implicit theorem, $K^* \in C^\infty (\Sigma')$.
Since $K^*$ satisfies (\ref{IKequS}),
\[ d \sigma(K^*)=d \frac{f_0}{\Lambda}
\]
and $\sigma(K^*(p_0))=\sigma(K_0)$, which means $K^*(p_0)=K_0$. Therefore $K^*$ is a solution of
(\ref{dKequS'}). By the uniqueness of the solutions of (\ref{IKequS}), $K^*$ is the
uniqueness solution of (\ref{dKequS'}). We prove the claim.
\end{proof}
In the following, we use $K$ to denote the solution of (\ref{dKequS}). Since each $q_l$,
$l=1,2,\cdots,L$, is a zero of $\omega$, $f_0$ can be smoothly extended to $q_l$ and $K$ can also
be smoothly extended to $q_l$ with $\mu < K(q_l) <-2\mu$. Next we have the following claim: \par

\noindent \textit{Claim 2:} $K$ can be continuously extended to $p_i$, $i=1,2,\cdots,I$, with
$K(p_i)=\mu$, to $q_{l'}$, $l'=L+1,L+2,\cdots,J$, with $K(q_{l'})=-2\mu$ and to $e_s$, $s=1,2,
\cdots,S$, with $K(e_s)=-2\mu$.
\begin{proof}
 First pick any $p_i$ and let $(U,z)$ be a \lccc around $p_i$ such that
$U \setminus \{p_i\} \subset \Sigma'$, $z(U)$ is a disk $D$
and $z(p_i)=0$. Then suppose
\BE\label{omegaequ1}
 \omega|_{U \setminus \{p_i\}}=\frac{\lambda_{-1}}{z}dz+d\eta_1=\frac{\eta_2(z)}{z}dz,
\EE
where $\lambda_{-1}$ is the residue of $\omega$ at $p_i$, $\eta_1$ is a holomorphic function on
$D$ and $\eta_2(z)$ is also a holomorphic function on $D$ with $\eta_2(0)=
\lambda_{-1}$. Then
\[ (\omega+\bar{\omega})|_{U \setminus \{p_i\}}=\lambda_{-1}d \ln|z|^2+2dRe(\eta_1)=
df_0.
\]
Therefore
\[f_0=\lambda_{-1}\ln|z|^2+2Re(\eta_1)+a^*~~\text{on}~D\setminus \{0\},
\]
where $a^*$ is a real constant, or equivalently,
\BE\label{f_0'}
 \frac{f_0}{\Lambda}=\frac{\lambda_{-1}}{\Lambda}\ln|z|^2+2Re(\frac{\eta_1}{\Lambda})+\frac{a^*}{\Lambda}~~
\text{on}~D \setminus \{0\}.
\EE
Substitute (\ref{f_0'}) into (\ref{IKequS}) to get on $D \setminus \{0\}$
\beyy
\ln(-2\mu-K)-\ln (K-\mu)-\frac{3\mu}{K-\mu}=\\
\frac{\lambda_{-1}}{\Lambda} \ln|z|^2+2Re(\frac{\eta_1}{\Lambda})+\frac{a^*}{\Lambda}+A_0,
\eeyy
where $A_0=\sigma(K_0)-\cfrac{f_0(p_0)}{\Lambda}$,
or equivalently, on $D \setminus \{0\}$
\BE\label{IKequp}
 (-2\mu-K)\frac{1}{K-\mu}e^{-\frac{3\mu}{K-\mu}}=A^* |z|^{\frac{2\lambda_{-1}}{\Lambda}}e^{
2Re(\frac{\eta_1}{\Lambda})},
\EE
where $A^*$ is a positive constant. Suppose that there is a sequence $\{z_n\}$ in $D
 \setminus \{0\}$ such that $z_n \rightarrow 0$ as $n \rightarrow \infty$ and
$\dps{\lim_{n \rightarrow \infty}}K(z_n)=b^*$. Then $\mu \leq b^* \leq -2\mu$. If $b^* \neq \mu$,
then substitute $\{z_n\}$ into (\ref{IKequp}), let $n \rightarrow \infty$ and take limits.
The limit of the left side is $(-2\mu-b^*)\cfrac{1}{b^*-\mu}e^{-\frac{3\mu}{b^*-\mu}}$,
the limit of the right side is $+\infty$(note $\lambda_{-1}>0$). It is a contradiction.
Hence $b^*=\mu$, which shows
\[ \lim_{z \rightarrow 0} K=\mu.
\]
\par
Similarly, we can prove that
\[ \lim_{p \rightarrow q_{l'}}K(p)=-2\mu~~\text{and}~~\lim_{p \rightarrow e_s}K(p)=-2\mu.
\]
Then we prove the claim.
\end{proof}
\noindent {\bf Step 2.} Define a metric on $\Sigma'$
\BE\label{gdef}
 g=-\frac{4}{3}(K-\mu)^2(K+2\mu)\omega\bar{\omega}.
\EE
{\textit{Claim 3.}} $g$ is an HCMU metric on $\Sigma'$ and $K$ is just the Gauss curvature of $g$.
\begin{proof}
Let $(U,z)$ be a local complex coordinate chart on $\Sigma'$. Suppose
\[\omega=\rho(z)dz~~\text{on}~U.
\]
Then
\[g|_U=-\frac{4}{3}(K-\mu)^2(K+2\mu)|\rho|^2|dz|^2.
\]
Let
\[e^{2\varphi}=-\frac{4}{3}(K-\mu)^2(K+2\mu)|\rho|^2,
\]
then
\[ \varphi=\frac{1}{2}\ln \frac{4(K-\mu)^2(-2\mu-K)|\rho|^2}{3}.
\]
Therefore
\BE\label{derivofphi}
\varphi_z=\frac{3\rho(K+\mu)K_z+(K+2\mu)(K-\mu)\rho_z}{2(K-\mu)(K+2\mu)\rho}.
\EE
By (\ref{dKequS}),
\BE\label{derivofK}
K_z=-\frac{1}{3}(K-\mu)^2(K+2\mu)\rho.
\EE
Substitute (\ref{derivofK}) into (\ref{derivofphi}) to get
\[ \varphi_z=\frac{1}{2}[-\rho (K^2-\mu^2)+\frac{\rho_z}{\rho}].
\]
Then
\[\varphi_{z\bar{z}}=-\rho KK_{\bar{z}}=\frac{1}{3}K(K-\mu)^2(K+2\mu)|\rho|^2.
\]
Therefore
\[
-\Delta \varphi=K e^{2\varphi},
\]
which shows $K$ is just the Gauss curvature of $g$. Meanwhile,
\[K_{\bar{z}}\rho=-\frac{1}{3}(K-\mu)^2(K+2\mu)|\rho|^2=\frac{1}{4}e^{2\varphi},
\]
so
\[ e^{-2\varphi}K_{\bar{z}}=\frac{1}{4\rho},
\]
which means $\nabla K$ is a holomorphic vector field on $\Sigma'$.
Hence $g$ is an HCMU metric on $\Sigma'$.
\end{proof}
\noindent {\bf Step 3.}~\textit{Claim~4}.~~$g$ is smooth at $e_s$, $s=1,2,\cdots,S$, and satisfies the angle
condition at $p_i$, $i=1,2,\cdots,I$, and $q_j$, $j=1,2,\cdots,J$. \par

\begin{proof}
Pick any $e_s$ and let $(U,z)$ be a local complex coordinate chart around $e_s$ such that
$U \setminus \{e_s\} \subset \Sigma'$, $z(U)$ is a disk $D$ and $z(e_s)=0$. Suppose that
\[ \omega|_{U \setminus \{e_s\}}=\frac{\Lambda}{z}dz+d\eta_1=\frac{\eta_2(z)}{z}dz,
\]
where $\eta_1$ is a holomorphic function on $D$ and $\eta_2(z)$ is a nonvanishing holomorphic function
on $D$. Then
\[g|_{U \setminus \{e_s\}}=e^{2\varphi}|dz|^2=-\frac{4}{3}(K-\mu)^2(K+2\mu)\frac{|\eta_2(z)|^2}
{|z|^2}|dz|^2.
\]
On the other hand, similar to (\ref{IKequp}), on $D \setminus \{0\}$
\[(-2\mu-K)\frac{1}{K-\mu}e^{-\frac{3\mu}{K-\mu}}=B_1 |z|^2e^{
2Re(\frac{\eta_1}{\Lambda})},
\]
where $B_1$ is a positive constant.
Then
\[ e^{2\varphi}=B_2(K-\mu)^3 e^{[2Re(\frac{\eta_1}{\Lambda})+\frac{3\mu}{K-\mu}]}|\eta_2(z)|^2,
\]
where $B_2=\cfrac{4}{3}B_1$. Therefore $e^{2\varphi}$ is continuous and positive on
$D$. Next consider the system of equations on $D \setminus \{0\}$:
\BE\label{equsys2}
 \begin{cases}
  K_{\bar{z}}=\cfrac{z}{4\eta_2(z)}e^{2\varphi}\\
  e^{2\varphi}=B_2(K-\mu)^3 e^{[2Re(\frac{\eta_1}{\Lambda})+\frac{3\mu}{K-\mu}]}|\eta_2(z)|^2.
 \end{cases}
\EE
Apply Lemma \ref{analysislemma} to the first equation of (\ref{equsys2}) to get
$K \in C^1(D)$. Then by the bootstrap technique, $K \in C^{\infty}(D)$ and $e^{2\varphi} \in
C^\infty(D)$. And
\[ -\Delta \varphi=K e^{2\varphi},~~K_{\bar{z}}=\frac{z}{4\eta_2(z)}e^{2\varphi}
\]
hold on $D$, which shows $g$ is actually an HCMU metric on $\Sigma^*=\Sigma \setminus \{
p_1,p_2,\cdots,p_I,q_1,q_2,\cdots,q_J\}$. \par

Pick any $p_i$ and let $(V,w)$ be a \lccc around $p_i$ such that $V \setminus \{p_i\} \subset \Sigma'$,
$w(V)$ is a disk $\widehat{D}$ and $w(p_i)=0$. Suppose
\[ \omega|_{V \setminus \{p_i\}}=\frac{\widehat{\lambda}_{-1}}{w}dw+d\widehat{\eta}_1=
\frac{\widehat{\eta}_2(w)}{w}dw,
\]
where $\widehat{\lambda}_{-1}=Res_{p_i}(\omega)$, $\widehat{\eta}_1$ is a holomorphic function
on $\widehat{D}$, $\widehat{\eta}_2(w)$ is also a holomorphic function on $\widehat{D}$ with
$\widehat{\eta}_2(0)=\widehat{\lambda}_{-1}$. Then on $\widehat{D} \setminus \{0\}$
\[ \ln(-2\mu-K)-\ln(K-\mu)-\frac{3\mu}{K-\mu}=\frac{2\widehat{\lambda}_{-1}}{\Lambda}\ln|w|+
2Re(\frac{\widehat{\eta}_1}{\Lambda})+\widehat{a},
\]
where $\widehat{a}$ is a real constant. Therefore
\BE\label{Kcusplim}
 \lim_{w \rightarrow 0}(K-\mu)\ln|w|=\widehat{A},
\EE
where $\widehat{A}$ is a nonzero real number.
On the other hand,
\[g|_{V \setminus \{p_i\}}=e^{2\widehat{\varphi}}|dw|^2=-\frac{4}{3}(K-\mu)^2(K+2\mu)
\frac{|\widehat{\eta}_2(w)|^2}{|w|^2}|dw|^2,
\]
that is,
\[e^{2\widehat{\varphi}}=-\frac{4}{3}(K-\mu)^2(K+2\mu)
\frac{|\widehat{\eta}_2(w)|^2}{|w|^2}.
\]
Then
\beyy
\widehat{\varphi} &=& \frac{1}{2}\ln[-\frac{4}{3}(K-\mu)^2(K+2\mu)\frac{|\widehat{\eta}_2(w)|^2}{|w|^2}]\\
&=& \ln(K-\mu)-\ln|w|+\frac{1}{2}\ln \frac{4(-2\mu-K)|\widehat{\eta}_2(w)|^2}{3}.
\eeyy
By (\ref{Kcusplim}),
\[ \lim_{w \rightarrow 0} \frac{\widehat{\varphi}+\ln|w|}{\ln |w|}=0,
\]
which shows $p_i$ is a cusp singularity of $g$. \par

Next pick any $q_l$ and let $(W,\xi)$ be a \lccc around $q_l$ such that $W \setminus \{q_l\} \subset \Sigma'$,
 $\xi(W)$ is a disk $D'$ and $\xi(q_l)=0$. Since $q_l$ is a zero of $\omega$ with $ord_{q_l}(\omega)=\alpha_l-1$,
suppose
\[ \omega|_W=\xi^{\alpha_l-1}H(\xi)d\xi,
\]
where $H(\xi)$ is a nonvanishing holomorphic function on $D'$. Therefore
\[g|_{W \setminus \{q_l\}}=-\frac{4}{3}(K-\mu)^2(K+2\mu)|\xi|^{2\alpha_l-2}|H(\xi)|^2
|d\xi|^2.
\]
Since $\mu < K(q_l)<-2\mu$, $g$ has a conical singularity at $q_l$ with the singular angle $2\pi \alpha_l$. \par

Finally pick any $q_{l'}$ and let $(Y,\zeta)$ be a \lccc around $q_{l'}$ such that $Y \setminus \{q_{l'}\}
\subset \Sigma'$ , $\zeta(Y)$ is a disk $\widetilde{D}$ and $\zeta(q_{l'})=0$. Suppose that
\[ \omega|_{Y \setminus \{q_{l'}\}}=\frac{\Lambda \alpha_{l'}}{\zeta}d\zeta+d\widetilde{\eta}_1=
\frac{\widetilde{\eta}_2(\zeta)}{\zeta}d\zeta,
\]
where $\widetilde{\eta}_1$ is a holomorphic function on $\widetilde{D}$ and
$\widetilde{\eta}_2(\zeta)$ is also a holomorphic function on $\widetilde{D}$ with
$\widetilde{\eta}_2(0)=\Lambda \alpha_{l'}$. Then on $\widetilde{D} \setminus \{0\}$
\[ \ln(-2\mu-K)-\ln(K-\mu)-\frac{3\mu}{K-\mu}=\alpha_{l'}\ln |\zeta|^2+2Re(\frac{\widetilde{\eta}_1}{\Lambda})+
\widetilde{a},
\]
where $\widetilde{a}$ is a constant.
On the other hand,
\BE\label{gcon2}
 g|_{Y \setminus \{q_{l'}\}}=-\frac{4}{3}(K-\mu)^2(K+2\mu)\frac{|\widetilde{\eta}_2(\zeta)|^2}{|\zeta|^2}
|d\zeta|^2.
\EE
Then substitute
\[ -2\mu-K=\widetilde{A}|\zeta|^{2\alpha_{l'}}(K-\mu)e^{[2Re(\frac{\widetilde{\eta}_1}{\Lambda})+\frac{3
\mu}{K-\mu}]}
\]
into (\ref{gcon2}) to get
\[ g|_{Y \setminus \{q_{l'}\}}=A^*(K-\mu)^3
e^{[2Re(\frac{\widetilde{\eta}_1}{\Lambda})+\frac{3\mu}{K-\mu}]}|\widetilde{\eta}_2(\zeta)|^2|
\zeta|^{2\alpha_{l'}-2}|d\zeta|^2,
\]
where $\widetilde{A}$ and $A^*$ are constants. Therefore
$g$ has a conical singularity at $q_{l'}$ with the singular angle $2\pi \alpha_{l'}$.
We prove the claim.
\end{proof}
\noindent {\bf Step 4.} $g$ has finite area and finite Calabi energy. \par
We can use the same method in calculating $\mathcal{C}_n$ to get $g$
has finite area and finite Calabi energy. \par
Hence we finish the proof of the sufficiency of Theorem \ref{mainth}.

\section{Existence of a kind of meromorphic 1-forms on a Riemann surface}
We see that the character 1-form $\omega$ of an HCMU metric on a compact Riemann surface
which has cusp singularities and conical singularities must have
the following properties:
\begin{itemize}
 \item[1.] $\omega$ only has simple poles,
 \item[2.] The residue of $\omega$ is a real number at each pole,
 \item[3.] $\omega+\bar{\omega}$ is exact on $\Sigma \setminus \{ poles~ of~\omega\}$.
\end{itemize}
In fact, on any Riemann surface(compact or noncompact), this kind of meromorphic 1-form exists.
\begin{theorem}[\cite{S}]\label{meroexi}
 Let $\Sigma$ be a Riemann surface and $p,q$ be two distinct points on $\Sigma$. Then there
exists a meromorphic 1-form $\omega$ on $\Sigma$ such that
\begin{itemize}
 \item[1.] $\omega$ only has two simple poles at $p$ and $q$ with $Res_p(\omega)=1$ and
$Res_q(\omega)=-1$;
 \item[2.] $\omega+\bar{\omega}$ is exact on $\Sigma \setminus \{p,q\}$.
\end{itemize}
\end{theorem}
By Theorem \ref{meroexi}, we can get the following theorem:
\begin{theorem}\label{meroexi'}
Let $\Sigma$ be a Riemann surface, $p_1,p_2,\cdots,p_n$ be $n(n\geq 2)$ points on $\Sigma$ and
$\lambda_1,\lambda_2,\cdots,\lambda_n$ be $n$ nonzero real numbers with $\dsum_{i=1}^n \lambda_i=0$.
Then there exists a meromorphic 1-form $\omega$ on $\Sigma$ such that
\begin{itemize}
 \item[1.]$\omega$ only has simple poles at $p_1,p_2,\cdots,p_n$ with $Res_{p_i}(\omega)=
\lambda_i$, $i=1,2,\cdots,n$;
 \item[2.] $\omega+\bar{\omega}$ is exact on $\Sigma \setminus \{p_1,p_2,\cdots,p_n\}$.
\end{itemize}
\end{theorem}
Now let $\Sigma$ be a compact Riemann surface. $\omega$ is a given meromorphic 1-form which
satisfies the conditions in Theorem \ref{meroexi'}. Then following the proof of the sufficiency
of Theorem \ref{mainth}, we can get there exists an HCMU metric which has cusp
singularities and conical singularities and just
has $\omega$ as the character 1-form. Meanwhile we can see it is possible that
different HCMU metrics(even HCMU metrics only with conical singularities and HCMU metrics with
cusp singularities and conical singularities) have the same character 1-form.


\par\vskip0.5cm
\noindent Qing Chen\\
Wu Wen-Tsun Key Laboratory of Mathematics,
USTC, Chinese Academy of
Sciences\\
School of Mathematical Sciences\\
University of Science and Technology of China\\
Hefei 230026 P. R. China\\
qchen@ustc.edu.cn\\
\vskip0.5cm
 \noindent Yingyi Wu \\
School of Mathematical Sciences \\
University of Chinese Academy of Sciences \\ Beijing 100049\\
P.R. China \\
wuyy@ucas.ac.cn\\
\vskip0.5cm
\noindent Bin Xu
\\ Wu Wen-Tsun Key Laboratory of Mathematics, USTC, Chinese Academy of
Sciences\\
School of Mathematical Sciences\\
University of Science and Technology of China\\
Hefei 230026 P. R. China
\end{document}